\documentclass[10pt]{amsart}

\title{The Quantum Bruhat Graph for $\widehat{SL}_2$ and Double Affine Demazure Products}
\author{Lewis Dean}

\usepackage[utf8]{inputenc}
\usepackage[margin=1.15in]{geometry}
\usepackage{amssymb}
\usepackage{amsmath}
\usepackage{amsthm}
\usepackage{mathtools, bm}
\usepackage{enumitem}
\usepackage{calc}
\usepackage{tikz-cd}
\usepackage{hyperref}
\usepackage{emptypage}
\usepackage{indentfirst}
\newtheorem{defn}{Definition}[section]
\newtheorem{thm}[defn]{Theorem}

\newtheorem{prop}[defn]{Proposition}
\newtheorem{cor}[defn]{Corollary}
 
\newtheorem{conj}[defn]{Conjecture} 
\theoremstyle{definition}

\theoremstyle{remark}
\newtheorem{rmk}[defn]{Remark}

\newcommand{\Tits}{\mathcal{T}}
\newcommand{\WTits}{W_\mathcal{T}}
\newcommand{\lev}{\textrm{\normalfont lev} }

\newcommand{\Phifin}{\Phi_{\textrm{\normalfont fin}}}
\newcommand{\Deltafin}{\Delta_{\textrm{\normalfont af}}}
\newcommand{\Wfin}{W_\textrm{\normalfont fin}}
\newcommand{\Waff}{W_\textrm{\normalfont aff}}
\newcommand{\Pfin}{P_\textrm{\normalfont fin}}
\newcommand{\Pdom}{P_\textrm{\normalfont dom}}

\newcommand{\half}{\frac{1}{2}}
\newcommand{\Inv}{\textrm{\normalfont Inv}}
\newcommand{\LP}{\textrm{\normalfont LP}}
\newcommand{\QBG}{\textrm{\normalfont QBG}}
\newcommand{\wt}{\textrm{\normalfont wt}}
\newcommand{\height}{\textrm{\normalfont ht}}
\newcommand{\sgn}{\textrm{\hspace{0.5mm}\normalfont sgn}}

\newcommand{\vbls}[1][1]{\vspace{{#1}\baselineskip}}
\newcommand{\numberthis}{\addtocounter{equation}{1}\tag{\theequation}}

\numberwithin{equation}{section}

\begin{document}

\begin{abstract}
We investigate the Demazure product in a double affine setting. Work by Muthiah and Pusk\'as gives a conjectural way to define this in terms of the $q=0$ specialisation of these Hecke algebras. We instead take a different approach generalising work by Felix Schremmer, who gave an equivalent formula for the (single) affine Demazure product in terms of the quantum Bruhat graph. We focus on type $\widehat{SL}_2$, where we prove that the quantum Bruhat graph of this type satisfies some nice properties, which allows us to construct a well-defined associative Demazure product for the double affine Weyl semigroup $\WTits$ (for level greater than one). We give results regarding the Demazure product and Muthiah and Orr's length function for $\WTits$, and we verify that our proposal matches specific examples computed by Muthiah and Pusk\'as using the Kac-Moody affine Hecke algebra.
\end{abstract}

\maketitle

\section{Introduction}
\subsection{Overview}

Affine Kac-Moody geometry has become a topic of great interest over recent years. Braverman, Kazhdan and Patnaik first constructed the Kac-Moody affine Hecke algebra \cite{BKP15} and gave it a $T$-basis indexed by a certain affine extension of the Weyl group, generalised further by Bardy-Panse, Gaussent and Rousseau \cite{BPGR16}. In the case that the underlying Weyl group is itself affine, this extension is called the \textit{double affine Weyl semigroup}, $\WTits$. 

$\WTits$ fails to be a Coxeter group, however much progress has been made in giving it Coxeter-like structures. Braverman, Kazhdan and Patnaik proposed a certain Bruhat order for $\WTits$ {\cite[Section~B]{BKP15}}, which was studied by Muthiah \cite{Mut18} who gave an equivalent definition in terms of a particular length function. This length function was further studied by Muthiah and Orr \cite{Muthiah-Orr19}, who also studied covers and cocovers in type ADE, further studied by Welch \cite{Welch19} and generalised to all types by Philippe \cite{philippe24}.

\vbls

Another structure of interest, which arises from Coxeter theory, is the Demazure product on a Weyl group $W$. This product is defined via the Bruhat order, but also describes multiplication in certain Kac-Moody Hecke algebras. It has many other interesting applications, describing the closure of the product of Iwahori double cosets \cite{He-Nie21}, as well as having uses in affine Deligne-Lusztig varieties \cite{He21}.

Recently, Muthiah and Pusk\'as have conjectured the existence of a Demazure product in $\WTits$ \cite{Muthiah-Puskas24}, arising from the corresponding Kac-Moody affine Hecke algebra. Our goal in this paper is to study an alternative approach to define this product, generalising work by Schremmer \cite{SchremABODP} to the double affine setting. This method relies on the (affine) quantum Bruhat graph, studied by Welch \cite{Welch19} for its uses with the Bruhat order, who generalised work by Milicevic \cite{Milicevic21}.

In this paper, we focus on $\WTits$ of type $\widehat{SL}_2$ and achieving results on the corresponding quantum Bruhat graph, and we end by matching up examples with those computed by Muthiah and Pusk\'as \cite{Muthiah-Puskas24}.

\subsection{Demazure Products}

First, we recall the combinatorial definition of the Demazure product. Consider the Coxeter group of either finite or affine type $W$ with simple roots $S$, length function $\ell(\cdot)$, and Bruhat order $\leq$. The \textit{Demazure product} of $x, y$, denoted $x*y$, is defined to be the maximal element with respect to the Bruhat order of the form $x'y'$, where $x' \leq x$ and $y' \leq y$. It has an equivalent iterative definition in terms of the simple reflections (cf. Equation \ref{eq:dem-def}). 

\vbls

The Demazure product also arises as the $q = 0$ specialisation of the Hecke algebra. Fix a $T$-basis $\{T_x\}_{x \in W}$ of the Hecke algebra, where $W$ is of finite or affine type. We then have the following:

\vbls[-0.5]

\begin{equation}\label{eqn:hecke-prod-mod}
    T_x T_y \equiv (-1)^{\ell(x) + \ell(y) - \ell(x*y)}T_{x*y}\ (\textup{mod}\ q), \quad x,y \in W
\end{equation}

In the affine case, this $T$-basis arises from the realisation of $W$ as a Coxeter group for some affine root system, however we can also write $W = \Wfin \ltimes \mathbb{Z}\Phi^{\vee}$, a semidirect product of the corresponding finite Weyl group $\Wfin$ with the coweight lattice $\mathbb{Z}\Phi^{\vee}$, which gives a Bernstein-type presentation.

\subsection{Kac-Moody Affine Hecke Algebras}

In the last few years, there has been significant progress in constructing affinisations of Kac-Moody algebras arising from considering the Kac-Moody group $G$ over $p$-adic fields. Let $G$ be an (untwisted) affine Kac-Moody group over a $p$-adic field $F$ with Weyl group $W$, now fixed to be of affine type. Let $G^+$ be the semigroup of $G$ where the Cartan decomposition holds. The full Kac-Moody affine Hecke algebra $\hat{\mathcal{H}}$ is the convolution algebra of $\mathbb{C}$-valued functions of $I \backslash G^+ / I$, defined by Braverman, Kazhdan and Patnaik \cite{BKP15} for $G$ untwisted affine, and by Bardy-Panse, Gaussent and Rousseau \cite{BPGR16} in the general case. 

They showed that there is a basis $\{T_x\}_{x \in \WTits}$, where the indexing set is the double affine Tits semigroup $W_\mathcal{T} = W \ltimes \mathcal{T}$, with respect to the Tits cone $\mathcal{T} = \bigcup_{w \in \Waff} wX$, where $X$ is the set of dominant coweights of $G$.

Note that, since $G$ is already affine, the Kac-Moody affine Hecke algebra is in fact double affine. We use the terms Kac-Moody affine and double affine synonymously when $G$ is itself affine. If instead $G$ were finite, then the corresponding Tits semigroup is exactly the affine Weyl group.

\vbls

For $T_x, T_y \in \hat{\mathcal{H}}$, we can write the product as:

\begin{equation}
    T_x T_y\ =\ \sum_{z \in W_\mathcal{T}} c^{z}_{x, y} T_z,\quad c_{x,y}^z \in \mathbb{Z}_{\geq 0} 
\end{equation}

Bardy-Panse, Gaussent and Rousseau \cite{BPGR16}, and independently Muthiah \cite{Mut18}, proved a conjecture by Braverman, Kazhdan and Patnaik \cite{BKP15}that the structure constants $c_{x,y}^z$ are integer-coefficient polynomials in $q$.

In their recent preprint \cite{Muthiah-Puskas24}, Muthiah and Pusk\'as further conjecture the following, a generalisation of equation \ref{eqn:hecke-prod-mod}:

\begin{conj}[Muthiah, Pusk\'as {\cite[Conjecture~7.3]{Muthiah-Puskas24}}]\label{conj:mp}
    There is exactly one coefficient $c_{x,y}^z \in \mathbb{Z}[q]$ in the product $T_xT_y$ which is non-zero modulo $q$. In particular, $T_x T_y \equiv (-1)^{\ell(x) + \ell(y) - \ell(z)}T_z\ (\text{\textup{mod}}\ q)$ for some $z \in W_\mathcal{T}$.
\end{conj}

\subsection{Schremmer's Work and the Quantum Bruhat Graph}

Conjecture \ref{conj:mp} gives one candidate for the Demazure product, but we instead take another approach of generalising work by Schremmer, who gave a more explicit formula for calculating the affine Demazure product in the extended affine Weyl group $\widehat{W}$. This approach relies on the (finite) quantum Bruhat graph and the notion of \textit{length positivity}. These concepts are easier to define and to work with than the Kac-Moody affine Hecke algebra, and generalise very naturally to the double-affine case. We state Schremmer's formula for the Demazure product below (Thm. \ref{schremform} in the main text).

\vbls

\begin{thm}
    Fix a (finite) Weyl group $W$ with coweight lattice $Q^{\vee}$. Construct the extended affine Weyl group $\widehat{W} = W \ltimes Q^{\vee}$, and let $x = w_x \varepsilon^{\mu_x}, y = w_y \varepsilon^{\mu_y} \in \widehat{W}$, where $w_x, w_y \in W$, $\mu_x, \mu_y \in Q^{\vee}$. Choose a pair of length positive elements $(u, v) \in W \times W$ such that $d(u \Rightarrow w_yv)$ is minimal amongst all such pairs, with some shortest path $p : u \Rightarrow w_yv$ in the quantum Bruhat graph $\QBG(W)$. Then the Demazure product $x*y$ is given as
    \begin{equation}\label{eqn:schremdeminit}
        x * y := w_x u v^{-1} \varepsilon^{vu^{-1} \mu_x + \mu_y - v\wt(p)}.
    \end{equation}
\end{thm}

To explain what this means, we first need to understand the quantum Bruhat graph for $W$, denoted $\QBG(W)$. This is some directed, weighted graph, which has vertices in $W$ and weights in $\Phi^{\vee}$, the coroots. The `upwards', or weight-zero edges, align with the Hasse diagram for $W$. The quantum Bruhat graph is then a particular extension of this, with downwards edges corresponding to quantum roots. 

The quantum Bruhat graph was formally constructed by Brenti, Fomin and Postnikov \cite{BFP98}, originally studied for its use in quantum cohomology \cite{Postnikov05}.

We remark that in the finite case, any two shortest paths have the same weight, and so we have a well-defined distance and weight function on paths in $\QBG(W)$. This was proven by Postnikov, and generalised to the parabolic quantum Bruhat graph by Lenart, Naito, Sagako, Schilling, and Shimozono, \cite{Len+15} who used the quantum Bruhat graph to study the affine Bruhat order, motivated by the study of Kirillov-Reshetikhin crystals \cite{Len+17}. 

\vbls

Furthermore, the notion of \textit{length positivity} in $W$ or a given element $x \in \widehat{W}$ is defined by some conditions on the length functional $\ell(x, u\alpha)$ (cf. Section \ref{sec:prelim}). We say that some $u \in \Wfin$ (the underlying finite Weyl group) is \textit{length positive} if $\ell(x, u\alpha)$ is non-negative for every positive root $\alpha$ (cf. Eqn 2.2). The set of length positive elements is then some subset of $W$, denoted $\LP(x)$.

\vbls

Note that the right hand side of equation \ref{eqn:schremdeminit} appears to depend on the choice of pair $(u,v)$, which is not necessarily unique, as well as the choice of path $p$. When considering the double-affine case, we will need to show that the generalised product we define is independent of these choices.

\subsection{Results for the Double Affine Weyl Semigroup}

It is easy to see that $W_\mathcal{T}$ is no longer a Coxeter group (it is no longer a group), however Muthiah and Orr, Welch, and Philippe have given it a well-defined length function and Bruhat order, both Coxeter-like structures. The previous conjecture gives us a candidate for a Demazure product in $\WTits$, suggesting that such a product should exist.

In the case of affine $SL_2$, we extend results regarding the finite quantum Bruhat graph to the affine one, finding the following theorem.

\begin{thm}[Thm. \ref{thm:same-weight}, Thm. \ref{thm:vertical-edge}]
    Let $W$ be the Weyl group for $\widehat{SL}_2$ with quantum Bruhat graph $\QBG(W)$. Let $u,v \in W$. Then:
    \begin{enumerate}[label=(\roman*)]
        \item Any two shortest paths from $u$ to $v$ have the same total weight.
        \item The length of a path between $u$ and an element of $\LP(x) \subseteq W$ for some $x \in \WTits$, in either direction, is minimised by a unique element of $\LP(x)$.
    \end{enumerate}
\end{thm}

This theorem, amongst others, leads us to the main focus of this paper.

\begin{thm}[Thm. \ref{thm:well-defined}, Thm. \ref{thm:assoc}]
    In type $\widehat{SL}_2$, Schremmer's formula for the Demazure product in $\WTits$ is well-defined for non-zero level, and associative for level greater than one.
\end{thm}

Furthermore, we prove that this product satisfies other important properties that we expect of a Demazure product in the general case, assuming well-definedness. In particular, we have the following results.

\begin{thm}[Thm. \ref{bigone}, Thm. \ref{demlen}]
    Consider any Kac-Moody root datum, and let $x, y \in \WTits$. Assume that the Demazure product $x*y$ is well-defined. Then:
    \begin{enumerate}[label=(\roman*)]
        \item $\ell(x*y) = \ell(x) + \ell(y)\ \iff\ \ell(xy) = \ell(x) + \ell(y)$, in which case $x*y = xy$.
        \item For an explicitly described path $p$ in $\QBG(W)$, $\ell(x*y) = \ell(x) + \ell(y) - \ell(p)$.
    \end{enumerate}
\end{thm}

In section \ref{sec:examples}, we go on to calculate some examples of Demazure products using Schremmer's formula, which match up with those computed by Muthiah and Pusk\'as from the Hecke algebra definition.

\subsection{Acknowledgements}

We thank Dinakar Muthiah for his supervision during this project, and many helpful discussions, as well as Dan Orr for comments on an earlier version of this paper. We also thank Cristian Lenart, Paul Philippe, and Anna Pusk\'as for interesting conversations on the topics included and their surrounding areas.

\vbls
\section{Notation and Preliminaries}\label{sec:prelim}
\subsection{Set-Up}

Fix a Kac-Moody root datum, and denote $\Phi$ for the set of roots. Choose a set of simple roots $\Delta \subset \Phi$, and write $\Phi^+$ for the set of positive roots. Let $P$ be the weight lattice, and $P^{\vee}$ the corresponding coweight lattice. Denote $\langle \cdot, \cdot \rangle$ for the natural pairing between $\mathbb{Z}\Phi$ and $P^{\vee}$. 

Let $W$ be the Weyl group, with its natural action on $\mathbb{Z}\Phi$. Write $e \in W$ for the identity element. Let $\Pdom$ be the dominant weights. The (integral) Tits cone $\Tits \subset P^{\vee}$ is defined to be the Weyl orbit of the dominant coweights, such that $\Tits = \bigcup_{w \in W}w(\Pdom^{\vee})$.

We construct $\WTits := W \ltimes \Tits$, the \textbf{affine Weyl semigroup} associated to $G$. In the case that $G$ is affine, we refer to $\WTits$ as the \textit{double} affine Weyl semigroup. A generic element $x \in \WTits$ is written as $x = w\varepsilon^{\mu}$, where $w \in W$ and $\mu \in P^{\vee}$. Note that in the finite case, $\Tits = P^{\vee}$, and so $\WTits = \widehat{W}$, the extended affine Weyl group.

\subsection{Demazure Product}

Consider $W$, or more generally any Coxeter group. For $w_1, w_2 \in W$, let $w_1 \leq w_2$ denote the Bruhat order. We define the \textbf{Demazure product} $*$ on $W$ to be the associative product satisfying the following, for $w, s \in W$ with $s$ a simple reflection.

\begin{equation} \label{eq:dem-def}
    w * s = \begin{cases}
        ws\ \ \text{if}\ ws > w, \\
        w\phantom{s}\ \ \text{if}\ ws < w.
    \end{cases}
\end{equation}

Note the similarity here to the relations satisfied by Demazure operators \cite{Demazure1973}. It is also well known that the Demazure product of two elements $w_1, w_2 \in W$ coincides with the set $\max\{w_1'w_2'\ |\ w_1' \leq w_1, w_2' \leq w_2\}$. This also holds in the extended affine Weyl group (see e.g. \cite{He09}).

Schremmer provides an alternate method of calculating the Demazure product in $\widehat{W}$ \cite{SchremABODP}, and therefore any finite or affine Weyl group, which we outline below with a goal of generalising the setup to a double affine $\WTits$.

\subsection{Quantum Bruhat Graph}

We introduce the quantum Bruhat graph, initially defined by Brenti, Fomin and Postnikov \cite{BFP98} for a finite Weyl group, and generalised to affine Weyl groups by Welch \cite{Welch19}. Let $\rho$ be the sum of the fundamental weights for $G$, equal to the half sum of the positive roots if finite. Let $\ell : W \rightarrow \mathbb{Z}_{\geq 0}$ denote the usual length function.

\begin{defn}
    The \textbf{quantum Bruhat graph} for $W$, denoted $\QBG(W)$, is the directed, weighted graph with vertex set $W$, weights in $\mathbb{Z}\Phi^{\vee}$, and an edge $w \rightarrow ws_{\alpha}$ for $\alpha \in \Phi^+$ if one of the following conditions is satisfied.
    \begin{enumerate}[label=\textup{(\roman*)}]
        \item $\ell(ws_{\alpha}) = \ell(w) + 1$.
        \item $\ell(ws_{\alpha}) = \ell(w) + 1 - \langle 2\rho, \alpha^{\vee} \rangle$.
    \end{enumerate}
    Edges of type \textup{(i)} are called \textbf{Bruhat edges} and have weight $0$. Edges of type \textup{(ii)} are \textbf{quantum edges} and have weight $\alpha^{\vee}$.
\end{defn}

We refer to Bruhat edges as ``upward" and quantum edges as ``downward", due to the length difference between the two relevant vertices. A \textbf{path} $p$ in $\QBG(W)$ is a sequence of consecutive edges. The \textbf{length} of a path is the number of edges, and its \textbf{weight} is the sum of the weights along those edges. We will also use the fact that $\langle 2\rho, \alpha^{\vee} \rangle = 2\height(\alpha)$.

\begin{prop}
    For $w_1, w_2 \in W$, there exists a path $w_1 \Rightarrow w_2$.
\end{prop}
\begin{proof}
    Following Postnikov's proof \cite{Postnikov05} for finite Weyl groups, let $w \neq e \in W$ and consider a simple reflection $s$ such that $ws < w$. Then $ws \rightarrow w$ is a Bruhat edge, and $w \rightarrow ws$ is a quantum edge. Iterating this, we find paths $w \rightarrow e$ and $e \rightarrow w$. Concatenating these paths for $w_1$ and $w_2$ respectively yields a path $w_1 \rightarrow w_2$.
\end{proof}

We say the \textbf{distance} between two vertices $d(w_1 \Rightarrow w_2)$ is the length of a shortest path between them. Note that $w_1 \leq w_2$ in the Bruhat order if and only if there exists a weight $0$ path between them. In the case that $W$ is finite, we also have the following:

\begin{prop}\textup{\cite[Lemma 1]{Postnikov05}}\label{postlem1}
    Assume $W$ is finite, and let $w_1, w_2 \in W$. Then all shortest paths from $w_1$ to $w_2$ have the same weight, denoted $\wt(w_1 \Rightarrow w_2)$.
\end{prop}

Postnikov's proof of Prop. \ref{postlem1} relies on \cite[Section 6]{BFP98}, who show that any two shortest paths $w_1 \rightarrow w_2$ can be obtained by applying weight-preserving switches of pairs of edges, however this does not hold in the affine case \cite{Welch19}.

\subsection{Schremmer's Formula}

Let $\Phi^+( \cdot )$ denote the indicator function for a positive root, returning 1 if the argument is positive, and 0 otherwise. We also write $\alpha > 0$ or $\alpha < 0$ if $\alpha$ is a positive or negative root respectively. The \textbf{length functional} for $x = w \varepsilon^{\mu} \in \WTits$ and $\alpha \in \Phi$ is defined as follows:

\vbls[-1]

\begin{equation}
    \ell(x, \alpha) = \langle \alpha, \mu \rangle + \Phi^+(\alpha) - \Phi^+(w \alpha).
\end{equation}

\vbls[0.5]

We say $v \in W$ is \textbf{length positive} for $x \in \WTits$ if $\ell(x, v\alpha) \geq 0$ for every $\alpha \in \Phi^+$. We write $\LP(x)$ for the set of length positive elements for $x$. 

\begin{defn}
    Let $x, y \in \WTits$, and let $w_y \in W$ be the Weyl component of $y$, so that $y = w_y \varepsilon^{\mu_y}$ for some $\mu_y \in P^{\vee}$. The \textbf{distance minimising set} for the pair $x,y$ is
    \begin{equation}
        M_{x,y} := \{ (u, v) \in \LP(x) \times \LP(y)\ |\ d(u \Rightarrow w_y v) \leq d(u' \Rightarrow w_y v')\ \forall u', v' \in W \}.
    \end{equation}
\end{defn}

We see that $M_{x,y}$ is always non-empty, as a shortest path between any two elements of $W$ always exists. The following definition is adapted from Thm. 5.11 of \cite{SchremABODP}.

\begin{defn}[Generalised Demazure Product]\label{def:gen-dem-prod}
    Let $x = w_x \varepsilon^{\mu_x}, y = w_y \varepsilon^{\mu_y} \in \WTits$, and let $(u,v) \in M_{x,y}$. Fix a shortest path $p : u \rightarrow w_y v$ in $\QBG(W)$. We define the \textbf{generalised Demazure product} $*_{u,v}^p$ on $\WTits$ as
    \begin{equation}
        x *_{u,v}^p y := w_x u v^{-1} \varepsilon^{vu^{-1} \mu_x + \mu_y - v\wt(p)}.
    \end{equation}
\end{defn}

Recall that in the finite case, $\wt(p)$ is independent on the choice of $p$, and so we write $*_{u,v} := *_{u,v}^p$. We have the following theorem due to Schremmer.

\begin{thm}\textup{\cite[Thm. 5.11, Prop. 5.12(b)]{SchremABODP}}\label{schremform}
    Assume $W$ is a finite Weyl group. Let $x,y \in \WTits = \widehat{W}$, and $(u,v) \in M_{x,y}$. 
    \begin{enumerate}[label=\textup{(\roman*)}]
        \item $x *_{u,v} y = x * y$, the Demazure product of $x$ and $y$.
        \item $uv^{-1}$ and $v\wt(u \Rightarrow w_y v)$ are independent of the choice of $(u,v) \in M_{x,y}$.
    \end{enumerate}
\end{thm}

This gives a method of calculating the Demazure product of two extended affine Weyl elements without directly using the Bruhat order, and instead considering $\QBG(W)$, which is finite. Our goal now is to generalise this to the double affine case, using it as a definition of a Demazure product in $\WTits$.

\begin{rmk}
    Schremmer proves Thm \ref{schremform}(i) by considering the Bruhat order on $\widehat{W}$ directly. Uniqueness of the Demazure product then immediately proves Thm \ref{schremform}(ii) as a corollary. \textit{A priori}, there is no direct proof of part (ii) that does not depend on the Demazure product.
\end{rmk}

\subsection{Double Affine Case}

Assume $W$ is an affine Weyl group, and let $\Wfin$ be the associated finite Weyl group with finite roots $\Phifin$, such that $W \cong \Wfin \ltimes \mathbb{Z}\Phifin^{\vee}$. Note that we can decompose the set of roots as $\Phi = \Phifin \times \mathbb{Z}\delta$, where $\delta$ is the minimal imaginary root. Let $\Deltafin$ be the finite simple roots, and let $\theta \in \Phifin$ be the longest root. Then $\Delta = \Deltafin \cup \{-\theta + \delta\}$.

We have positive roots $\Phi^+ = \{\tilde{\alpha} = \sgn(n)(\alpha + n\delta)\ |\ \alpha \in \Phifin^+, n \in \mathbb{Z} \}$, where $\sgn(n) = 1$ for $n \geq 0$, and is $0$ otherwise. We also have the decomposition $P^{\vee} = \Pfin^{\vee} \oplus \mathbb{Z}\delta \oplus \mathbb{Z}\Lambda_0$, where $\Pfin$ is the weight lattice of $\Wfin$, and $\Lambda_0$ is a fundamental (affine) coweight chosen such that, for $\tilde{\beta} \in \Delta$, $\langle \tilde{\beta}, \Lambda_0 \rangle = 0$ and $\langle \delta, \Lambda_0 \rangle = 1$. 

\begin{prop}\textup{\cite[Prop 5.8(b)]{KacIDLA}}
    \begin{equation}
        \Tits = \{\lambda + m\delta + l\Lambda_0\ |\ \lambda \in \mathbb{Z}\Phifin^{\vee}, m \in \mathbb{Z}, l \in \mathbb{Z}_{> 0} \}  \cup  \mathbb{Z}\delta.
    \end{equation} 
\end{prop}

Given $\mu = \lambda + m\delta + l\Lambda_0 \in \Tits$ and $x = w\epsilon^\mu$, we say that the level of $\mu$ (or the level of $x$) is $\lev(\mu) = \lev(x) = l$. Note that all elements of $\Tits$ have non-negative level, and those with zero level are a multiple of $\delta$. We say that an element of the form $m\delta \in \Tits$ is on the \textit{boundary} of the Tits cone.

A generic element $x \in \WTits$ is written as $x = w \varepsilon^{\mu} = w_0 \tau^{\lambda} \varepsilon^{\mu}$ for some $w \in W, \mu \in P^{\vee}, w_o \in \Wfin, \lambda \in \Pfin^{\vee}$. Recall the action of $W$ on $\Phi$ induced by the standard action of $\Wfin$ on $\Phifin$, given explicitly as

\vbls[-0.5]

\begin{equation}\label{eqn:w-action}
    w_0 \tau^{\lambda}( \alpha + n\delta) = w_0\alpha + (n - \langle \alpha, \lambda \rangle) \delta.
\end{equation} 

\vbls[-0.5]

\subsection{\texorpdfstring{$\mathbf{\widehat{SL}_2}$}{}}

Our goal is to show that Schremmer's formula defines a product on $\WTits$, resembling the Demazure product, which is independent of the choice of $(u,v) \in M_{x,y}$ and the path between them. We present a proof for type $\widehat{SL}_2$, which we fix for the rest of the paper unless stated otherwise.

We have $\Phifin = \{ \pm \alpha \}$ and $\Wfin = \langle s_1\ |\ s_1^2 = e \rangle \cong S_2$. The roots are $\Phi = \{\sgn(\alpha + n\delta)\ |\ n \in \mathbb{Z} \}$ with Weyl group $W = \langle s_0, s_1\ |\ s_0^2 = s_1^2 = e \rangle \cong \widetilde{S}_2$, the infinite Dihedral group. We also have the presentation $W \cong \Wfin \ltimes \mathbb{Z}\Phifin^{\vee} = \langle w_0 \tau^{\lambda}\ |\ w_0 \in \Wfin, \lambda \in \mathbb{Z}\alpha^{\vee} \rangle$, linked by the relation $s_0 = s_1 \tau^{-\alpha^{\vee}}$.
\vbls
\section{Classifying Length Positive Elements}
\subsection{The Length Functional}

Let $x,y \in \WTits$. In order to prove that the product $x *_{u,v}^p y$ as defined in equation \ref{schremform} is independent of the choice of $(u,v) \in M_{x,y}$, we need to be able to describe the length positive sets $\LP(x), \LP(y)$ from which $M_{x,y}$ is defined. To do this, we fix an arbitrary $x \in \WTits$ and expand on the conditions required for $u \in W$ to be length positive for $x$.

Let $x = w\varepsilon^{\mu} = w_0 \tau^{r\alpha^{\vee}} \varepsilon^{\mu}$ where $\mu = k\alpha^{\vee} + m\delta + l\Lambda_0$. Let $v = v_0 \tau^{t \alpha^{\vee}} \in \LP(x) \subseteq W$ for some $v_0 \in \Wfin, t \in \mathbb{Z}$. We seek the conditions on $v$ that force $v \in \LP(x)$. Fix $\tilde{\alpha} = \sgn(n)(\alpha + n\delta) \in \Phi^+$ as above for positive roots. Let $\sgn(v_0)$ denote the parity of $v_0$, such that $\sgn(e) = 1$ and $\sgn(s_1) = -1$. Expanding the action of $W$ on $\Phi$, we obtain the following:

\vbls[-1]

\begin{align}
    v\tilde{\alpha}
    &= \sgn(n)(v_0 \alpha + (n - 2t)\delta),\\
    wv\tilde{\alpha} 
    &= \sgn(n)( w_0 v_0 \alpha + (n - 2t - 2r\sgn(v_0))\delta),\\
    \langle \mu, v\tilde{\alpha} \rangle
    &= \sgn(n)(2k \sgn(v_0) + (n-2t)l). \label{expansion:bracket}
\end{align}

Each of these appear as terms in the length functional that we wish to consider:

\vbls[-0.5]

\begin{equation}\label{lfv}
     \ell(x, v\tilde{\alpha}) 
     = \langle \mu, v\tilde{\alpha} \rangle + \Phi^+(v\tilde{\alpha}) - \Phi^+(wv\tilde{\alpha})
     \geq 0,\ \ \forall \tilde{\alpha} \in \Phi^+.
\end{equation}

\begin{rmk}
    The indicator component of equation \ref{lfv}, $\Phi^+(v\tilde{\alpha}) - \Phi^+(w\tilde{\alpha})$, has absolute value at most 1. Hence, we must have $\langle \mu, v\tilde{\alpha} \rangle \geq -1$ for all $\tilde{\alpha} \in \Phi^+$. In particular, if we have $\langle \mu, v\tilde{\alpha} \rangle \geq 1$ for every positive root $\tilde{\alpha}$, then we immediately find that $v \in \LP(x)$.
\end{rmk}

The above remark suggests that we should split the calculation into cases concerning the minimum value attained by $\langle \mu, v\tilde{\alpha}\rangle$ as $\tilde{\alpha}$ varies over $\Phi^+$, denoted $\min\{\langle \mu, v\tilde{\alpha} \rangle\}_{\tilde{\alpha}}$. We deal with the boundary case separately, as $\mu$ behaves differently. Our cases to consider are:

\begin{enumerate}[label=(\alph*)]
    \item $l > 0, \
    \begin{cases}
        \text{(i)\phantom{ii}}\ \min\{\langle \mu, v\tilde{\alpha} \rangle\}_{\tilde{\alpha}} \geq 1. \\
        \text{(ii)\phantom{i}}\ \min\{\langle \mu, v\tilde{\alpha} \rangle\}_{\tilde{\alpha}} = 0. \\
        \text{(iii)}\ \min\{\langle \mu, v\tilde{\alpha} \rangle\}_{\tilde{\alpha}} = -1.
    \end{cases}$
    \item $l = 0\ $ (boundary case).
\end{enumerate}

In each case, we want to find equivalent conditions in terms of $t$ and $v_0$ which define $v$. It will be convenient to instead first find conditions on $\sgn(v_0)k - tl$, where $k, l, t, v_0$ are defined as before.

\subsection{Case (ai)} We require $\langle \mu, v\tilde{\alpha} \rangle \geq 1$ for any $\tilde{\alpha} = \sgn(n)(\alpha + n\delta) \in \Phi^+$. Each root is given by the corresponding integer $n$, and so we can equivalently enforce this condition for every $n \in \mathbb{Z}$. From equation \ref{expansion:bracket}, this condition becomes

\begin{equation}\label{condition:1}
    \sgn(n)(2k\sgn(v_0) + (n-2t)l) \geq 1,\ \ \forall n \in \mathbb{Z}.
\end{equation}

Note that if \ref{condition:1} holds for $n=0$, then it holds for all $n \geq 0$. Similarly, if it holds for $n=-1$, then it must hold for all $n \leq -1$. Hence, we can reduce this to just two inequalities.

\vbls[-0.5]

\begin{align}\label{condition:2}
    2k\sgn(v_0) -2tl &\geq 1, \\
    -2k\sgn(v_0) + (1+2t)l &\geq 1.
\end{align}

\vbls[-1]

\begin{equation}\label{eqn:inequality-ai}
    \implies\quad \half\ \leq\ k\sgn(v_0) - tl\ \leq\ \frac{l-1}{2}.
\end{equation}

\vbls[0.5]

Therefore, $\sgn(v_0)k - tl \in \{1, 2, ..., \lfloor \frac{l-1}{2} \rfloor\}$ in order to have $v \in \LP(x)$.

\subsection{Case (aii)} We now require $\langle \mu, v\tilde{\alpha} \rangle \geq 0$ for any $\tilde{\alpha} = \sgn(n)(\alpha + n\delta) \in \Phi^+$, attained for at least one root. As in the previous case, the inequalities for all $n$ can be reduced to just checking the cases $n=0, -1$, and so we get the following:

\begin{equation}
    0\ \leq\ k\sgn(v_0) - tl\ \leq\ \frac{l}{2}.
\end{equation}

Note the $n=0$ and $n=1$ cases correspond to the left and right hand inequalities respectively. One of these must be attained to ensure that $\min\{\langle \mu, v\tilde{\alpha}\rangle\}_{\tilde{\alpha}} = 0$, and so we get that $k \in \{\sgn(v_0)tl, \sgn(v_0)(\frac{l}{2} + tl)\}$ for $l$ even, and $k = \sgn(v_0)tl$ for $l$ odd. However, we need to check that the length positivity condition is also satisfied:

\vbls[-1]

\begin{align*}
    \ell(x, v\tilde{\alpha}) &= \langle \mu, v\tilde{\alpha} \rangle + \Phi^+(v\tilde{\alpha}) - \Phi^+(wv\tilde{\alpha}) \\
    &= \Phi^+(v\tilde{\alpha}) - \Phi^+(wv\tilde{\alpha})\ \geq\ 0
\end{align*}

\vbls[-1]

\begin{equation}
    \iff \Phi^+(v\tilde{\alpha}) = 1\ \ \text{or}\ \ \Phi^+(wv\tilde{\alpha}) = 0 .
\end{equation}

Expanding our form for $v\tilde{\alpha}$, we find that $\Phi^+(v\tilde{\alpha}) = 1 \iff \sgn(n)(v_0\alpha + (n - 2t)\delta) > 0$, where $n \in \{0,-1\}$ by the above. If $n=0$, then in order to ensure that $v_0\alpha - 2t\delta > 0$, we must have that $t \leq \half (\Phi^+(v_0 \alpha) - 1)$, or equivalently $t \leq \Phi^+(v_0\alpha) - 1$. If $n = -1$, we similarly obtain $t \geq \half (\Phi^+(v_0 \alpha) - 1)$, equivalent to $t \geq 0$.

Applying the same method to the requirement $\Phi^+(wv\tilde\alpha) = 0$, where $w = w_0\tau^{r\alpha^\vee}$, for $n=0$, $t$ must satisfy $t \geq \Phi^+(w_0v_0\alpha) - r\sgn(v_0)$, and for $n=-1$, $t$ must satisfy $t \leq -1 -r\sgn(v_0)$.

\vbls

Summarising this case, we find that $v \in \LP(x)$ if one of the following holds:

\begin{align}
    &\sgn(v_0)k - tl = 0,\ \textrm{and either}\ t \leq \Phi^+(v_0\alpha) - 1\ \textrm{or}\ t \geq \Phi^+(w_0v_0\alpha) - r\sgn(v_0) \label{rmk:case-aii-1} \\
    &\sgn(v_0)k - tl = \frac{l}{2}\ \textrm{(with $l$ even),\ and either}\ t \geq 0\ \textrm{or}\ t \leq -1 - r\sgn(v_0). \label{rmk:case-aii-2}
\end{align}

\subsection{Case (aiii)} Continuing on as before, we require $\langle \mu, v\tilde{\alpha} \rangle \geq -1$, attained for at least one root, which leads to the following double-sided inequality.

\begin{equation}
    -\half\ \leq\ \sgn(v_0)k - tl\ \leq\ \frac{l+1}{2}
\end{equation}

This must be attained on one side, and so we must have that $\sgn(v_0)k - tl = \frac{l+1}{2}$, with $l$ odd, corresponding to $n=-1$. For $n=0$, length positivity is ensured, since $\sgn(v_0)k - tl = \frac{l+1}{2} \geq \frac{1}{2}$ and so $\langle \mu, v\tilde{\alpha}\rangle \geq 1$ in this case. Checking the length positivity condition for $n=-1$, we must enforce:

\vbls[-0.5]

\begin{align*}
     &\Phi^+(v\tilde{\alpha}) - \Phi^+(wv\tilde{\alpha}) \geq 1 \\
     \iff\ &\Phi^+(v\tilde{\alpha}) = 1\ \text{and}\ \Phi^+(wv\tilde{\alpha}) = 0.
\end{align*}

Using calculations done in case (aii), this occurs for $n=-1$ when $0 \leq t \leq -1 - r\sgn(v_0)$, giving the requirements for $v \in \LP(x)$ in this case.

\subsection{Case (b)} Now let $l = 0$, and so $\mu \in \mathbb{Z}\delta$, giving us that $k = 0$ and hence $\langle \mu, v\tilde{\alpha} \rangle = 0$. The length positivity condition then becomes the same as in case (aii), namely that either $\Phi^+(v\tilde{\alpha}) = 1$ or $\Phi^+(wv\tilde{\alpha}) = 0$, however now we must have that at least one of these holds for every $n \in \mathbb{Z}$. Instead of looking at inequalities as above, we can classify the length positive set in this case using inversion sets, defined for $u \in W$ as $\Inv(u) = \{ \tilde{\alpha} \in \Phi^+\ |\ u\tilde{\alpha} < 0 \}$. We use these to prove a simple description of $\LP(x)$.

\begin{prop}\label{prop:boundary-lp}
    Let $x = w\varepsilon^{m\delta} \in \WTits$ be a boundary element of the Tits cone, where $w \in W$ and $m \in \mathbb{Z}$. If $w \neq e$, write $w = w's_i$, with $w' < w$ and $i \in \{0,1\}$. Then
    \begin{equation}
        \LP(x) = 
        \begin{cases}
            \{ v \in W\ |\  s_iv > v \}, & \text{if}\ w \neq e. \\
            W, & \text{otherwise.}
        \end{cases}
    \end{equation}
\end{prop}

\begin{proof}
Note that $v \notin \LP(x)$ if and only if $\Phi^+(v\tilde{\alpha}) = 0$ and $\Phi^+(wv\tilde{\alpha}) = 1$ for some $\tilde{\alpha} \in \Phi^+$. This occurs if and only if  $v\tilde{\alpha} < 0$ and $w(-v\tilde{\alpha}) < 0$, equivalent to $\tilde{\alpha} \in \Inv(v)$ and $-v\tilde{\alpha} \in \Inv(w)$, or $-v\tilde{\alpha} \in -v\Inv(v) \cap \Inv(w)$.  Hence $v \notin \LP(x) \iff \exists \tilde{\alpha} \in \Inv(w) \cap -v\Inv(v) \iff \Inv(w) \cap -v\Inv(v) \neq \emptyset$. Negating this statement, and noting that from \cite{Humphreys_1990} we have $-v\Inv(v) = \Inv(v^{-1})$, we find the description $\LP(x) = \{v \in W\ |\ \Inv(w) \cap \Inv(v^{-1}) = \emptyset \}$. 
We have that $\Inv(w) \cap \Inv(v^{-1})\ \iff\ \ell(wv) = \ell(v) + \ell(w)$, i.e. $v, w$ are length additive, and so we have

\begin{equation}
    \LP(x) = \{v \in W\ |\ \ell(wv) = \ell(v) + \ell(w)\}.
\end{equation}

\vbls[0.5]

Note that this equation holds for any general $\WTits$, and not necessarily just the case of $\widehat{SL}_2$. First note that if $w = e$, then this condition holds for all $v \in W$, and so $\LP(x) = W$. 

If we now fix $\WTits$ to be of type $\widehat{SL}_2$ with $w \neq e$, then we have an easy description of length additivity. In particular, as there are no braid relations in $W$, we have that $v$ and $w$ are length additive if and only if, in reduced expressions for $v$ and $w$, $v = s_j v'$ and $w = w's_i$ with $v > v', w > v'$, and $j \neq i$. This occurs if and only if $s_iv > v$, proving the claim.
\end{proof}

\begin{rmk}\label{rmk:boundary-lp-details}
    We can describe $\LP(w\varepsilon^{m\delta})$ as consisting of the half of $W$ whose first simple reflection in a reduced expression is not the last simple reflection in a reduced expression for $w$. In particular, it is clear that $|\LP(x)| = \infty$.
\end{rmk}

\begin{cor}
    Let $x = w\epsilon^{m\delta} \in \WTits$. Then $\LP(x)$ is connected in $\QBG(W)$.
\end{cor}

This can be seen graphically in Fig. \ref{fig:full-QBG}, where $\LP(x)$ consists of an entire left or right side, which are connected by vertical edges.

\subsection{Algorithm for computing \texorpdfstring{$\LP(x)$}{}}

We can combine the results from above to give a short algorithm for computing $\LP(x)$. Firstly, if $l = 0$, we use Prop. \ref{prop:boundary-lp}, giving us $\LP(x)$ immediately. However, we are more interested in the non-boundary cases. Here we state and prove an algorithm using the above results for computing $\LP(x)$.

\begin{prop}\label{prop:LP-algorithm}
    Let $x = w_0\tau^{r\alpha^{\vee}}\epsilon^{k\alpha^{\vee} + m\delta + l\Lambda_0} \in \WTits$, with $l > 0$. Then $\LP(x)$ can be computed using the following algorithm.
    \begin{enumerate}
        \item Compute $t \in \mathbb{Z}$ such that $-\frac{l}{2} < k - tl \leq \frac{l}{2}$, and write $j = k - tl$.
        \item If $1 \leq \sgn(v_0)j \leq \frac{l-1}{2}$ for some choice of $v_0 \in \Wfin$, then $\tau^{t\alpha^{\vee}}v_0 \in \LP(x)$.
        \item If $j = 0$, then $\tau^{t\alpha^{\vee}} \in \LP(x)$ if either $t \leq 0 $ or $t \geq \Phi^+(w_0\alpha) - r$, and $\tau^{t\alpha^{\vee}}s_1 \in \LP(x)$ if either $t \geq 1$ or $t \leq \Phi^+(w_0\alpha) - r - 1$.
        \item If $j = \frac{l}{2}$, then $\tau^{t\alpha^{\vee}} \in \LP(x)$ if either $t \geq 0$ or $t \leq -1-r$, and $\tau^{t\alpha^{\vee}}s_0 \in \LP(x)$ if either $t \leq -1$ or $t \geq -r$.
        \item If $j = -\sgn(v_0)\frac{l-1}{2}$ for some choice of $v_0 \in \Wfin$, then $\tau^{t\alpha^{\vee}}s_0 \in \LP(x)$ if $v_0 = s_1$ and $-r \leq t \leq -1$, and $\tau^{t\alpha^{\vee}}s_1s_0 \in \LP(x)$ if $v_0 = e$ and $1 \leq t \leq -r$.
    \end{enumerate}
    After applying step 5, we have found all elements of $\LP(x)$.
\end{prop}

\newpage 

\begin{proof}
    We first show that each step of the algorithm holds true. 
    \begin{enumerate}
        \item Such a $t$ exists, and is unique by the Euclidean algorithm.
        \item If $j > 0$, then this immediately holds by equation \ref{eqn:inequality-ai}, with $v_0 = e$. If $j < 0$, then $-j = -k - (-t)l$, and so we apply equation $\ref{eqn:inequality-ai}$ with $v_0 = s_1$, using $-t$ instead of $t$.
        \item Follows from \ref{rmk:case-aii-1}, noting $\sgn(v_0)j = 0$ regardless of the choice of $v_0$, and if we choose $v_0 = s_1$ then we need to replace $t$ with $-t$ using the same logic as in the proof of 2. 
        \item Follows from \ref{rmk:case-aii-2}, noting that if $j = \frac{l}{2}$ then $-j = -\frac{l}{2}$, and so $-k - (-t-1)l = \frac{l}{2}$ also satisfies the required conditions with $v_0 = e$ and using $-(t+1)$ instead of $t$.
        \item Follows from Case (aiii) above, applying the same logic as in step 4 to shift $t$ to meet the required conditions.
    \end{enumerate}
    As $t$ is unique, we can get at most one length positive element from each case (for a particular $v_0$). Hence, as we have covered all possible values of $j$ in the given range, we must have found all elements of $\LP(x)$.
\end{proof}

Below, we show two examples of length positive sets in type $\widehat{SL}_2$, corresponding to the elements $x = \tau^{-\alpha^{\vee}}\varepsilon^{-2\alpha^{\vee}+\delta+4\Lambda_0}$ and $y = s_1\tau^{-\alpha^{\vee}}\varepsilon^{\alpha^{\vee} - \delta + \Lambda_0}$ respectively. Using the above algorithm, it can be computed that $\LP(x) = \{s_1, s_1s_0\}$ and $\LP(y) = \{e, s_0, s_0s_1\}$. More detailed calculations of some length positive sets are given in section \ref{sec:examples}.

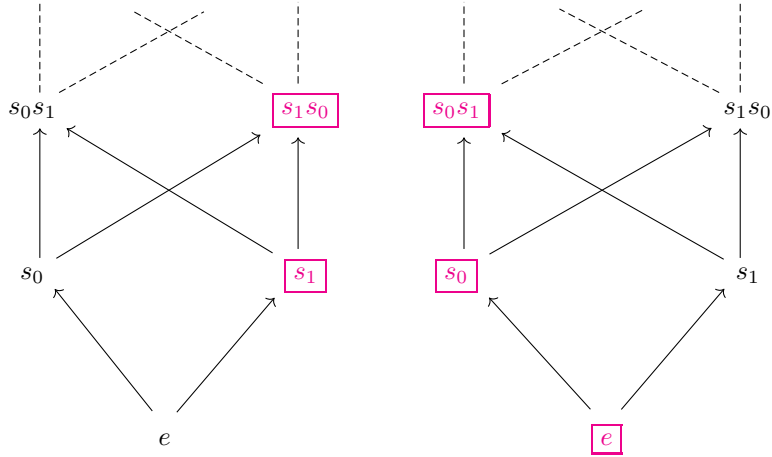
\begin{figure}[ht]
    \centering
    \begin{tikzcd}[row sep=1.5cm, column sep=1.05cm]
        \ && \ \\
        s_0 s_1
        \arrow[u, dashed, dash, shorten >= 0.4cm, shift right = 0.1cm] 
        \arrow[urr, dashed, dash, shorten >= 1.3cm]
        & & \textcolor{magenta}{\boxed{s_1 s_0}}
        \arrow[ull, dashed, dash, shorten >= 1.3cm] 
        \arrow[u, dashed, dash, shorten >= 0.4cm, shift left = 0.1cm] \\
        s_0
        \arrow[u, shift right = 0.1cm] 
        \arrow[urr]
        & & \textcolor{magenta}{\boxed{s_1}}
        \arrow[ull] 
        \arrow[u, shift left = 0.1cm] \\
        & e \vphantom{\boxed{e}}
        \arrow[ul, shift right = 0.1cm]
        \arrow[ur, shift left = 0.1cm]
    \end{tikzcd}
    \qquad 
    \begin{tikzcd}[row sep=1.5cm, column sep=1.05cm]
        \ && \ \\
        \textcolor{magenta}{\boxed{s_0 s_1}}
        \arrow[u, dashed, dash, shorten >= 0.4cm, shift right = 0.1cm] 
        \arrow[urr, dashed, dash, shorten >= 1.3cm]
        & & s_1 s_0
        \arrow[ull, dashed, dash, shorten >= 1.3cm] 
        \arrow[u, dashed, dash, shorten >= 0.4cm, shift left = 0.1cm] \\
        \textcolor{magenta}{\boxed{s_0}}
        \arrow[u, shift right = 0.1cm] 
        \arrow[urr]
        & & s_1
        \arrow[ull] 
        \arrow[u, shift left = 0.1cm] \\
        & \textcolor{magenta}{\boxed{e}}
        \arrow[ul, shift right = 0.1cm]
        \arrow[ur, shift left = 0.1cm]
    \end{tikzcd}     
    \caption{For $W$ of type $\widehat{SL}_2$, the length positive sets $\LP(\tau^{-\alpha^{\vee}}\varepsilon^{-2\alpha^{\vee}+\delta+4\Lambda_0})$ and $\LP(s_1\tau^{-\alpha^{\vee}}\varepsilon^{\alpha^{\vee} - \delta + \Lambda_0})$ respectively, each depicted within the Hasse diagram for $\widehat{SL}_2$.}
    \label{fig:LP-examples}
\end{figure}

\vbls[0.5]

This classification of $\LP(x)$ for each $x \in \WTits$ gives us some immediate corollaries.

\begin{cor}\label{cor:LPx-size}
    Let $x \in \WTits$ with $\lev(x)>0$. Then $1 \leq |\LP(x)| \leq 2$ if $\lev(x) \neq 1$, and $1 \leq |\LP(x)| \leq 3$ if $\lev(x) = 1$. 
\end{cor}

\begin{proof}
    Fix $j = k - tl$ as in Prop. \ref{prop:LP-algorithm}, and first assume that $l \neq 1$. Then if $j \notin \{0, \frac{l}{2}, \pm\frac{l-1}{2}\}$, step 2 gives us exactly one element of $\LP(x)$.
    
    If $j = 0$, then we certainly have a length positive element if $t \leq 0$. If instead $t \geq 0$, assume that $t \leq -(\Phi^+(w_0\alpha) - r)$. Then $\Phi^+(w_0\alpha) - r \leq 0$, and so $t \geq \Phi^+(w_0\alpha) - r$. Hence at least one condition is satisfied. There are also clearly at most two possibilities for length positive elements.
    
    If $j = \frac{l}{2}$, then similarly to before we have a length positive element if $t \geq 0$. If instead $t \leq -1$, and $t \geq -r$, then we must have $r \leq 1$ and so $t \leq -1 + r$ is guaranteed. Hence, at least one condition is satisfied, and again there are at most two possibilities.
    
    If $j = \pm\frac{l-1}{2}$, then we are guaranteed a length positive element from case 2, and can only include at most one more element from this final step, giving at most two cases.

    If instead $l = 1$, then we find no length positive elements from step 2 of Prop. \ref{prop:LP-algorithm}, however we have $j = 0 = \frac{l-1}{2}$. Step 3 returns either 1 or 2 elements depending on the values of $t, r$, and step 5 can return at most 1 more element, giving $1 \leq |\LP(x)| \leq 3$.  
\end{proof}

\begin{cor}\label{cor:LPx-connected}
    Let $x \in \WTits$ with $\lev(x) > 0$. Then $\LP(x)$ is connected in $\QBG(W)$ by edges of the form $w \rightarrow ws_i$.
\end{cor}
\begin{proof}
    Firstly, we take $|\LP(x)| > 1$, otherwise the statement is true trivially. 
    
    If $|\LP(x)| = 2$, then following the algorithm, the only possible pairs of elements we find are sets of the form  $\{\tau^{t\alpha^{\vee}}, \tau^{t\alpha^{\vee}}s_1\},\{\tau^{t\alpha^{\vee}}, \tau^{t\alpha^{\vee}}s_0\}$, and $\{\tau^{t\alpha^{\vee}}s_1, \tau^{t\alpha^{\vee}}s_1s_0\}$, each of which are of the form $\{w, ws_i\}$ for $w \in W,\ i \in \{0, 1\}$, which are connected in $\QBG(W)$ by a single edge.

    If $|\LP(x)| = 3$, then we must have $l = 1$ and $j = 0 = \pm\frac{l-1}{2}$, forcing $\{\tau^{t\alpha^{\vee}}, \tau^{t\alpha^{\vee}}s_1\} \subset \LP(x)$ and either $\tau^{t\alpha^{\vee}}s_0 \in \LP(x)$ or $\tau^{t\alpha^{\vee}}s_1s_0 \in \LP(x)$. Hence, the triple of elements we find are connected in $\QBG(W)$ either as $\tau^{t\alpha^{\vee}}s_0 \rightarrow \tau^{t\alpha^{\vee}} \rightarrow \tau^{t\alpha^{\vee}}s_1$ or as $\tau^{t\alpha^{\vee}} \rightarrow \tau^{t\alpha^{\vee}}s_1 \rightarrow \tau^{t\alpha^{\vee}}s_1s_0$ as required.
\end{proof}

\begin{cor}\label{cor:no-degen-3lp}
    Let $x \in \WTits$ with $|\LP(x)| = 3$. Then $\LP(x) \neq \{s_1, e, s_0\}$.
\end{cor}

\begin{proof}
    If $|\LP(x)| = 3$, then $\lev(x) = 1$. Hence, using the notation of Prop. \ref{prop:LP-algorithm}, $j = 0$ and so $\LP(x) \subset \{\tau^{t\alpha^{\vee}}, \tau^{t\alpha^{\vee}}s_1, \tau^{t\alpha^{\vee}}s_0, \tau^{t\alpha^{\vee}}s_1s_0\}$. For $e \in \LP(x)$, we therefore require $t = 0$ or $t = 1$. If $t = 1$, then $s_1 \notin \LP(x)$. If $t = 0$, then step 5 of the classification algorithm guarantees that $s_0 \notin \LP(x)$.
\end{proof}
\vbls
\section{Generalised Demazure Product for \texorpdfstring{$\widehat{SL}_2$}{}}
\subsection{Quantum Bruhat Graph of type \texorpdfstring{$\widehat{SL}_2$}{}}

The main goal of this section is to show that the product $x *_{u,v}^p y$ is independent of the choice of $u, v$ and $p$ for $x, y \in \WTits$, $G = \widehat{SL}_2$. Firstly, we investigate the quantum Bruhat graph for $W$, $\QBG(W)$. This was studied by Amanda Welch \cite{Welch19}, who gave the following description of the edges.

\begin{prop}
    Let $W = \langle s_0, s_1\ |\ s_0^2 = s_1^2 = e \rangle$ be the Weyl group of type $\widehat{SL}_2$. Then the upwards edges are of the form $v \rightarrow vs_{\tilde{\alpha}}$ for $\tilde{\alpha} \in \Phi^+$, where $\ell(vs_{\tilde{\alpha}}) = \ell(v) + 1$, and the downwards edges are of the form $v \rightarrow vs_{\tilde{\alpha}}$ for $\tilde{\alpha} \in \Delta = \{\alpha, \alpha[-1]\}$, where $\ell(vs_{\tilde{\alpha}}) = \ell(v) - 1$.
\end{prop}

We depict this graph in Fig. \ref{fig:full-QBG}, with upwards edges depicted in black with zero weight, and downwards edges in magenta labelled with their weight. Note that downwards edges are purely vertical (excepting at $e$), as they only appear associated to reflections with simple affine roots. Following this graphical description, we refer to an element $w \in W$ as being \textit{left-sided} if $s_1w > w$, and \textit{right-sided} if $s_0w > w$. Note that the identity $e$ is both left and right-sided. We call an edge $u \rightarrow v$ \textit{vertical} if $u, v$ are same-sided, and call it \textit{diagonal} $u, v$ are opposite-sided. Note again that paths to and from $e$ are simultaneously vertical and diagonal. In this language, we immediately find the following proposition from Corollaries \ref{cor:LPx-connected}, \ref{cor:no-degen-3lp}.

\begin{prop}
    Let $x \in \WTits$ with $\lev(x) > 0$. The $\LP(x)$ is either entirely left-sided, or entirely right-sided.
\end{prop}

\begin{rmk}
    For $W$ of general type, upwards (Bruhat) edges in $\QBG(W)$ are easy to construct, corresponding to covers in the Bruhat order. On the other hand, downwards (quantum) edges, which correspond to \textit{quantum roots}, are seemingly much more difficult to compute. We mention that H\'ebert and Philippe have recently studied these roots in more detail \cite{Hebert-Philippe24}.
\end{rmk}

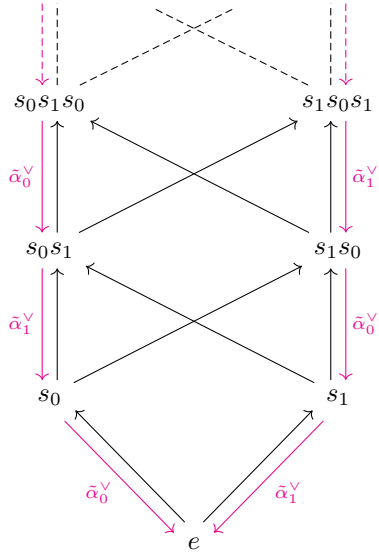
\begin{figure}
    \centering
    \begin{tikzcd}[row sep=1.5cm, column sep=1.05cm]
        \
        \arrow[d, magenta, shift right = 0.1cm, dashed, shorten <= 0.4cm]
        && \ 
        \arrow[d, magenta, shift left = 0.1cm, dashed, shorten <= 0.4cm] \\
        s_0 s_1 s_0
        \arrow[u, dashed, dash, shorten >= 0.4cm, shift right = 0.1cm] 
        \arrow[urr, dashed, dash, shorten >= 1.3cm]
        \arrow[d, magenta, shift right = 0.1cm, "\tilde{\alpha}_0^{\vee}"']
        & & s_1 s_0 s_1
        \arrow[ull, dashed, dash, shorten >= 1.3cm] 
        \arrow[u, dashed, dash, shorten >= 0.4cm, shift left = 0.1cm]
        \arrow[d, magenta, shift left = 0.1cm, "\tilde{\alpha}_1^{\vee}"] \\
        s_0 s_1
        \arrow[u, shift right = 0.1cm] 
        \arrow[urr]
        \arrow[d, magenta, shift right = 0.1cm, "\tilde{\alpha}_1^{\vee}"']
        & & s_1 s_0
        \arrow[ull] 
        \arrow[u, shift left = 0.1cm]
        \arrow[d, magenta, shift left = 0.1cm, "\tilde{\alpha}_0^{\vee}"] \\
        s_0
        \arrow[u, shift right = 0.1cm] 
        \arrow[urr]
        \arrow[dr, magenta, shift right = 0.1cm, "\tilde{\alpha}_0^{\vee}"']
        & & s_1
        \arrow[ull] 
        \arrow[u, shift left = 0.1cm]
        \arrow[dl, magenta, shift left = 0.1cm, "\tilde{\alpha}_1^{\vee}"] \\
        & e 
        \arrow[ul, shift right = 0.1cm]
        \arrow[ur, shift left = 0.1cm]
    \end{tikzcd}
    \caption{$\QBG(W)$ for $W$ of type $\widehat{SL}_2$.}
    \label{fig:full-QBG}
\end{figure}

\vbls 

\subsection{Shortest Paths in \texorpdfstring{$\QBG(W)$}{}}
We first aim to show that $*_{u,v}^p = *_{u,v}$, i.e. that the generalised Demazure product for $\widehat{SL}_2$ is independent of the path $p$ chosen. We instead prove an alternate statement, showing that the weights of any two shortest paths must be equal, from which path-independence follows.

\begin{thm}\label{thm:same-weight}
    Let $u, v \in W$ of type $\widehat{SL}_2$ and let $p_1, p_2:u\rightarrow v$ be any two shortest paths between them. Then $\wt(p_1) = \wt(p_2)$.
\end{thm}

\begin{proof}
    First, note that upwards edges correspond to a length increase of one, and downwards edges correspond to a length decrease of one. Hence, for a path $p:u\rightarrow v$, we have that $\ell(p) = |\ell(v) - \ell(u)|$ if and only if $p$ consists of solely upwards or downwards edges, in which case it must be shortest. Note also that if we can find such a path, then all shortest paths must be of the same form.
    
    If $u \leq v$, then by definition there must be a path consisting of only upwards edges (arising form the corresponding Hasse diagram for $W$), and hence all shortest paths have this property. The weight of any such path is zero, and so all paths have the same weight.

    If $u \geq v$, and $u,v$ are same-sided, then there exists a unique shortest path consisting of downwards vertical edges.

    If none of the above cases hold, $u,v$ are not same-sided and $u \nleq v$. Then it is clear from the diagram that a shortest path consists of vertical downwards edges and one upwards diagonal edge. Fix one such path $p$, and consider the path $p'$ consisting of the same edges, but with the diagonal taken one edge earlier (assuming this path is still shortest). Then the weight of the subpath being removed is $\tilde{\alpha}_i^{\vee} + 0$ for some $i \in \{0,1\}$, and the weight of the subpath being added is $0 + \tilde{\alpha}_i^{\vee}$ for the same $i$, as shown in the graph. Hence, $\wt(p') = \wt(p) - (\tilde{\alpha}_i^{\vee} + 0) + (\tilde{\alpha}_i^{\vee} + 0) = \wt(p)$. A similar argument applies if we consider the path taken with one edge later, and iterating this shows that all shortest paths must have the same weight.
\end{proof}

We henceforth write $\wt(u \Rightarrow v)$ for the weight of any shortest path $u \rightarrow v$. From this, we immediately obtain the following corollary.

\begin{cor}
    For type $\widehat{SL}_2$, the generalised Demazure product $*_{u,v}^p = *_{u,v}$ is independent of $p$.
\end{cor}

The following fact about paths in $\QBG(W)$ will also be useful when showing that the generalised Demazure product is independent of the choice of $(u,v) \in M_{x,y}$.

\begin{thm}\label{thm:vertical-edge}
    Let $x \in \WTits, \lev(x) > 0$, and let $u \in W$. Then $\ell(u \Rightarrow v)$ is minimised by a unique $v \in \LP(x)$, as is $\ell(v \Rightarrow u)$.
\end{thm}

\begin{proof}
    As before, if $|\LP(x)| = 1$ then the statement is trivial, so assume that $v, vs_i \in \LP(x)$.
    
    By the proof of Thm. \ref{thm:same-weight}, $\ell(u \Rightarrow v) = |\ell(u) - \ell(v)|$ if either $u \leq v$ or $u \geq v$ and $u,v$ are same-sided, or $\ell(u \Rightarrow v) = |\ell(u) - \ell(v)| + 2$ otherwise. It is then clear that $\ell(u \Rightarrow v) \neq \ell(u \Rightarrow vs_i)$, as these values must have different parities, since $\ell(vs_i) = \ell(v) \pm 1$, and so one of $v, vs_i$ will uniquely minimise the distance among that pair, proving the theorem in the case $|\LP(x)| = 2$.
    
    Finally assume $\LP(x) = \{v, vs_i, vs_is_j\}$ for $i\neq j$. Note that by Cor. \ref{cor:no-degen-3lp}, $\LP(x)$ is entirely one-sided. It is clear from Fig. \ref{fig:full-QBG} that $\ell(u \Rightarrow v) = \ell(u \Rightarrow vs_is_j)$ if and only if either $u = vs_i$, in which case $\ell(u \Rightarrow vs_i) = 0$ and $vs_i$ uniquely minimises the distance, or if $u \rightarrow vs_i$ is a diagonal edge, in which case $\ell(u \Rightarrow vs_i) = 1$ and $vs_i$ again uniquely minimises the distance. Hence, either $\ell(u \Rightarrow v), \ell(u \Rightarrow vs_i), \ell(u \Rightarrow vs_is_j)$ are pairwise distinct, and so we have a unique distance minimiser, or $\ell(u \Rightarrow v) = \ell(u \Rightarrow vs_is_j)$, in which case $vs_i$ is the unique distance minimiser, proving this part of the theorem in the last case.

    An analogous argument proves the theorem for $\ell(v \Rightarrow u)$.
\end{proof}

\subsection{Independence of choice in \texorpdfstring{$M_{x,y}$}{}}

We now state the main theorem of this section.

\begin{thm}\label{thm:well-defined}
    For type $\widehat{SL}_2$, the generalised Demazure product $x *_{u,v}^p y = x*y$ is independent of the choice of $(u,v) \in M_{x,y}$ and $p:u \Rightarrow v$.
\end{thm}

Before proving this, we first show its equivalence to a more explicit condition.

\begin{prop}
    Let $x,y \in \WTits, (u,v) \in M_{x,y}$, and let $w_y \in W$ be the Weyl component of $y$. Then if $uv^{-1}$ and $v\wt(u \Rightarrow w_y v)$ are independent of the choice of $u, v$, theorem \textup{\ref{thm:well-defined}} holds.
\end{prop}

\begin{proof}
    Note that by Thm. \ref{thm:same-weight}, $*_{u,v}^p$ is certainly independent of $p$. Write $uv^{-1} = \varphi_1, v\wt(p) = v\wt(u \Rightarrow w_y v) = \varphi_2$. Then the generalised Demazure product becomes
    \begin{align*}
        x *_{u,v}^p y 
        = x *_{u,v} y
        &= w_x u v^{-1} \varepsilon^{vu^{-1} \mu_x + \mu_y - v\wt(u \Rightarrow w_y v)} \\
        &= w_x \varphi_1 \varepsilon^{\varphi_1^{-1} \mu_x + \mu_y - \varphi_2}.
    \end{align*}
    Hence if $\varphi_1, \varphi_2$ are independent of the choice of $u,v$,  the product is fully determined.
\end{proof}

\begin{proof}[Proof of Theorem \textup{\ref{thm:well-defined}}] 
We first prove this holds for $x, y \in \WTits$ with $\lev(x), \lev(y) \neq 0$.

By Cor \ref{cor:LPx-size}, we have that $1 \leq |\LP(x)|, |\LP(y)| \leq 3$, and by Cor \ref{cor:LPx-connected}, $\LP(x), \LP(y)$ are each connected by vertical edges in $\QBG(W)$. Note also that $w\LP(y)$ will also be connected by vertical edges, as multiplication on the left preserves this condition.

If either $|\LP(x)| = 1$ or $|\LP(y)| = 1$, then by Thm. \ref{thm:vertical-edge}, there is a unique distance minimising pair $M_{x,y} = \{(u, v)\}$, and so there are no choices to be made, in which case the theorem is true immediately.

\vbls

Next we consider $|\LP(x)| = |\LP(y)| > 1$, with either $|\LP(x)| = 2$ or $|\LP(y)| = 2$. By Thm. \ref{thm:vertical-edge}, we must have that $|M_{x,y}| \leq 2$. Assume that $\LP(x) \cap w_y\LP(y) = \emptyset$, otherwise a distance minimising pair $(u,v)$ must have $u = w_y v$, in which case $uv^{-1} = w_y$ and $v\wt(u \Rightarrow w_y v) = 0$ are both forced.

If $\LP(x)$ and $w_y\LP(y)$ are same-sided in $\QBG(W)$, then it is clear from that graph that there is a unique distance minimising pair. If instead $\LP(x)$ and $w_y\LP(y)$ are opposite-sided, then there is again a unique distance minimising pair unless 2 elements of $\LP(x)$ are each connected to an element of $w_y\LP(y)$ by a diagonal edge, as pictured in Fig. \ref{fig:qbg-portion}. In this case, $|M_{x,y}| = \{(u,v), (us_i, vs_i)\}$, where $\{u, us_i\} \subseteq \LP(x)$ and $\{v, vs_i\} \subseteq \LP(y)$, with $u < us_i$ and $v < vs_i$.

\vbls

\begin{figure}
\begin{center}
\begin{tikzcd}[row sep=1.5cm, column sep=1.05cm]

\
&& \ 
\arrow[d, magenta, shift left = 0.1cm, dashed, shorten <= 0.4cm] \\
\
\arrow[d, magenta, shift right = 0.1cm, dashed, dash, shorten <= 0.4cm]
& & w_y v s_i
\arrow[ull, dashed, dash, shorten >= 1.8cm] 
\arrow[u, dashed, dash, shorten >= 0.4cm, shift left = 0.1cm]
\arrow[d, magenta, shift left = 0.1cm, "\tilde{\alpha}_i"] \\
u s_i
\arrow[u, dashed, dash, shift right = 0.1cm, shorten >= 0.4cm] 
\arrow[urr]
\arrow[d, magenta, shift right = 0.1cm, "\tilde{\alpha}_i"']
& & w_y v
\arrow[ull, dashed, dash, shorten >= 1.8cm] 
\arrow[u, shift left = 0.1cm]
\arrow[d, magenta, shift left = 0.1cm, dashed, dash, shorten >= 0.4cm] \\
u
\arrow[u, shift right = 0.1cm] 
\arrow[urr]
\arrow[d, magenta, shift right = 0.1cm, dashed, dash, shorten >= 0.4cm]
& & \
\arrow[ull, dashed, dash, shorten <= 1.8cm] 
\arrow[u, shift left = 0.1cm, dashed, dash, shorten <= 0.4cm] \\
\
\arrow[u, shift right = 0.1cm, dashed, dash, shorten <= 0.4cm]
& & \\

\end{tikzcd}
\end{center}
\vspace{-4\baselineskip}
\caption{A portion of $\QBG(W)$ when $\ell(w_y v) = \ell(us_i)$ and $w_y v, us_i$ are opposite-sided.}\label{fig:qbg-portion}
\end{figure}
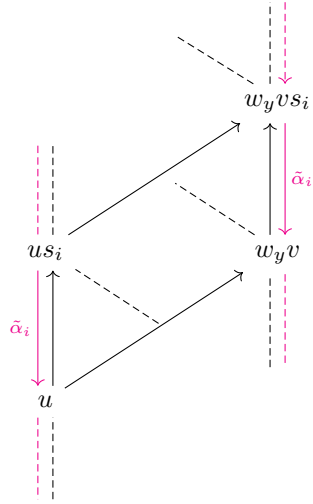

\newpage

For each pair $(u_1, v_1) \in M_{x,y}$, we calculate $\varphi_1 = u_1v_1^{-1}$ and $\varphi_2 = v_1\wt(u_1 \Rightarrow w_yv_1)$:

\vbls[-1]

\begin{alignat*}{4}
    &\varphi_1:\quad
    \begin{aligned}[t]
        &(us_i)(w_yvs_i)^{-1}
        &&=\ us_i s_i v^{-1} w_y^{-1}
        =\ u v^{-1} w_y^{-1}.\\
        &(u)(w_yv)^{-1}
        &&=\ uv^{-1} w_y^{-1}. \\
    \end{aligned} \\[0.5\baselineskip]
    &\varphi_2:\quad
    \begin{aligned}[t]
        &vs_i\wt(us_i \Rightarrow w_yvs_i)
        &&=\ vs_i(0)
        &=\ 0.\\
        &v\wt(u \Rightarrow w_yv)
        &&=\ v(0)
        &=\ 0.   
    \end{aligned}
\end{alignat*}

Hence $\varphi_1$ and $\varphi_2$ are independent of the choice of pair $(u_1, v_1) \in M_{x,y}$ as required.

\vbls

The final non-boundary case to consider is when $|\LP(x)| = |\LP(y)| = 3$. If $|M_{x,y}| = 1$, then we are done. If $|M_{x,y}| = 2$, then we must have a situation as in Fig. \ref{fig:qbg-portion}, in which case the previous calculations apply and we are again done. The final option is $|M_{x,y}| = 3$, in which case $M_{x,y} = \{(u, v), (us_i, vs_i), (us_is_j, vs_is_j)\}$, for $i \neq j$, where $u,v$ are opposite-sided and $\ell(v) = \ell(u) + 1$.
We have proven above that $\varphi_1, \varphi_2$ are constant for pairs of the form $(u,v), (us_i, vs_i)$ as previously described, and so must also be constant for $(us_i, vs_i), (us_is_j, vs_is_j)$. Hence, these values are constant for every pair in $M_{x,y}$, as required.

\vbls

Now we must consider the case when at least one of $x, y$ are on the boundary of the Tits cone. If $\lev(x) = 0$ but $\lev(y) > 0$, recall that $\LP(x)$ is either a complete left or right side of $\QBG(W)$ by Prop. \ref{prop:boundary-lp}, and $w_y\LP(y)$ is connected vertically. Hence, these sets are either same-sided or opposite-sided. In the first case, we must have $w_y\LP(y) \subset \LP(x)$, and so $\LP(x) \cap w_y \LP(y) \neq \emptyset$, forcing the choices of $\varphi_1$ and $\varphi_2$. If instead they are opposite-sided, then it is easy to see that a distance minimising path must have length 1, and we recover the same situation as in Fig. \ref{fig:qbg-portion}, with either 2 or 3 distance minimising pairs. An analogous argument applies if $\lev(x) > 0 $ and $\lev(y) = 0$.

Finally, if $\lev(x) = \lev(y) = 0$, then $w_y\LP(y)$ is single-sided, consisting of all elements above $w_y$. If $\LP(x)$ is same-sided, then we have an overlap and we are done, and if $\LP(x)$ is opposite-sided, we see that all distance minimising pairs are again of the same form as in Fig. \ref{fig:qbg-portion}, in which case $\varphi_1$ and $\varphi_2$ are constant.
\end{proof}

\subsection{Determining \texorpdfstring{$\LP(x*y)$}{}}

In order to show associativity further on, we will need to prove an extension of Theorem 5.12(b) in \cite{SchremABODP} to the double-affine case. We write $M_{x,y}^y = \{v \in \LP(y)\ |\ (u,v) \in M_{x,y}$ for some $u \in \LP(x)\}$, the $\LP(y)$ component of the set of distance-minimising pairs. Firstly, we state the following proposition, an extension of a lemma by Schremmer (\cite{SchremGNP}, Lemma 2.14).

\begin{prop}\label{prop:lp-iff-zero}
    Let $\tilde{\beta}$ be a simple affine root, and let $x \in \WTits$ with $u \in \LP(x)$. Then $us_{\tilde{\beta}} \in \LP(x) \iff \ell(x, u\tilde{\beta}) = 0$.
\end{prop}
\begin{proof}
    Mimicking Schremmer's proof for the affine case, observe that the condition $u, us_{\tilde{\beta}} \in \LP(x)$ implies  $\ell(x, u\tilde{\beta}), \ell(x, -u\tilde{\beta}) \geq 0 \implies \ell(x, u\tilde{\beta}) = 0$. Furthermore, if $\ell(x, u\tilde{\beta}) = 0$, then $\ell(x, u\tilde{\alpha}) \geq 0$ for any $\tilde{\alpha} \in \Phi^+ \cup \{-\tilde{\beta}\} = s_{\tilde{\beta}}(\Phi^+ \cup \{-\tilde{\beta}\})$, and so $\ell(x, us_{\tilde{\beta}}\tilde{\alpha}) \geq 0$ for all $\tilde{\alpha} \in \Phi^+$ as required.
\end{proof}

\begin{thm}\label{thm:LP-of-demprod}
    Let $x, y \in \WTits$ such that $\lev(x), \lev(y) \geq 1$. Then $\LP(x*y) = M_{x,y}^y$.
\end{thm}

\begin{proof}
    Let $(u,v) \in M_{x,y}$, and write $x = w_x \varepsilon^{\mu_x}, y = w_y \varepsilon^{\mu_y}$ as before. Let $r \in W$. Then:
    \begin{align*}
        \ell(x*y, r\tilde{\alpha})\
        &=\ \langle vu^{-1}\mu_x + \mu_y - v\wt(u \Rightarrow w_y v), r\tilde{\alpha} \rangle + \Phi^+(r\tilde{\alpha}) - \Phi^+(w_xuv^{-1}r\tilde{\alpha})\\
        &=\ \left( \langle \mu_y, r\tilde{\alpha} \rangle + \Phi^+(r \tilde{\alpha}) - \Phi^+(w_y r \tilde{\alpha}) \right) + \left(\langle \mu_x, uv^{-1}r\tilde{\alpha} \rangle + \Phi^+(uv^{-1}r\tilde{\alpha}) - \Phi^+(w_xuv^{-1}r\tilde{\alpha})\right) \\
        &\hspace{50mm} - \langle \wt(u \Rightarrow w_yv), v^{-1}r\tilde{\alpha} \rangle + \Phi^+(w_yr\tilde{\alpha}) - \Phi^+(uv^{-1}r\tilde{\alpha})\\
        &=\ \ell(y, r\tilde{\alpha}) + \ell(x, uv^{-1}r\tilde{\alpha}) - \langle \wt(u \Rightarrow w_yv), v^{-1}r\tilde{\alpha} \rangle + \Phi^+(w_yr\tilde{\alpha}) - \Phi^+(uv^{-1}r\tilde{\alpha}).
    \end{align*}

    Assume $r \neq v$. Our goal is to show that $r$ cannot be length positive for $x*y$ unless it is part of another distance-minimising pair. Write $r = vr_0$ for some $v_0 \in W \backslash \{e\}$, so
    \begin{equation}
        \ell(x*y, r\tilde{\alpha})\
        =\ \ell(y, vr_0\tilde{\alpha}) + \ell(x, ur_0\tilde{\alpha}) - \langle \wt(u \Rightarrow w_yv), r_0\tilde{\alpha} \rangle + \Phi^+(w_yvr_0\tilde{\alpha}) - \Phi^+(ur_0\tilde{\alpha}).
    \end{equation}

    \vbls[0.5]

    Let $\tilde{\alpha} \in \Inv(r_0)$ and write $\tilde{\beta} = -r_0\tilde{\alpha} \in \Inv(r_0^{-1})$. Using the properties that $\ell(z, -\tilde{\alpha}) = -\ell(z, \tilde{\alpha})$ and $\Phi^+(-\tilde{\alpha}) = 1 - \Phi^+(\tilde{\alpha})$ (cf. \cite{SchremGNP}), we find:
    \begin{equation}
        \ell(x*y, r\tilde{\alpha})\
        =\ -\ell(y, v\tilde{\beta}) - \ell(x, u\tilde{\beta}) + \langle \wt(u \Rightarrow w_yv), \tilde{\beta} \rangle - \Phi^+(w_yv\tilde{\beta}) + \Phi^+(u\tilde{\beta}).
    \end{equation}

    It is clear from the graph of $\QBG(W)$ that the weights of edges that contribute to $\wt(u \Rightarrow w_yv)$ alternate between $\tilde{\alpha}_0 = -\alpha + \delta$ and $\tilde{\alpha}_1 = \alpha$. Hence, writing $\wt(u \Rightarrow w_yv) = c_1\alpha + c_2\delta$, we must have $c_1 \in \{0, \pm 1\}$ and $c_2 \in \mathbb{Z}_{\geq 0}$. Therefore we must have that $\langle \wt(u \Rightarrow w_yv), \tilde{\beta} \rangle \in \{-2, 0, 2\}$. 
    
    Note that $\ell(y, v\tilde{\beta}), \ell(x, u\tilde{\beta}) \geq 0$ as $u \in \LP(x), v \in \LP(y)$. Hence, if $\langle \wt(u \Rightarrow w_yv), \tilde{\beta} \rangle = -2$ for some $\tilde{\beta} \in \Inv(r_0^{-1})$, then $\ell(x*y, r\tilde{\alpha})\ \leq\ 0 + 0 -2 - 0 + 1\ =\ -1$, and so $r \notin \LP(x*y)$.
    
    \vbls
    
    Assume that $r \in \LP(x*y)$. We first show that if $\langle \wt(u \Rightarrow w_yv), \tilde{\beta} \rangle = 2$ for some $\tilde{\beta} \in \Inv(r_0^{-1})$, then $|\Inv(r_0^{-1})| = 1$. 
    
    If $|\Inv(r_0^{-1})| > 1$, write $r_0^{-1} = ...s_j s_i$, for some $i,j \in \{0,1\}, i \neq j$. Then by the classification of inversion sets (\cite{Humphreys_1990}), we must have that $\{\tilde{\alpha}_i, s_i(\tilde{\alpha}_j)\} \subset \Inv(r_0^{-1})$. It is easy to check that the classical parts of these two roots will differ by a sign. Then if $\langle \wt(u \Rightarrow w_yv), \tilde{\beta} \rangle = 2$ for some $\tilde{\beta} \in \Inv(r_0^{-1})$, it must also be non-zero for any $\tilde{\beta} \in \Inv(r_0^{-1})$. Hence, the product of the weight with $\tilde{\alpha}_i$ and $s_i(\tilde{\alpha}_j)$ must both be non-zero, and also differ by a sign. We can therefore choose $\tilde{\beta} \in \{\tilde{\alpha}_i, s_i(\tilde{\alpha}_j)\}$ such that $\langle \wt(u \Rightarrow w_yv), \tilde{\beta} \rangle = -2$, implying that $r \notin \LP(x*y)$ by the above, giving a contradiction.

    Therefore, we must have that either $\langle \wt(u \Rightarrow w_yv), \tilde{\beta} \rangle = 2$ with $r_0 = s_{\tilde{\beta}}$, or $\langle \wt(u \Rightarrow w_yv), \tilde{\beta} \rangle = 0$ for all $\tilde{\beta} \in \Inv(r_0^{-1})$.

    \vbls

    In the first case, note that $\langle \wt(u \Rightarrow w_yv), \tilde{\beta} \rangle = 2$ implies that we must have an odd number of downwards edges in a shortest path $u \Rightarrow w_yv$, and $\tilde{\beta}$ will be the simple root that occurs most often as a weight in this path. From the diagram of $\QBG(W)$, we then see that we must have $u\tilde{\beta} < 0$ and $w_yv\tilde{\beta} > 0$, and so $\ell(x*y, r\tilde{\alpha}) = -\ell(y, v\tilde{\beta}) - \ell(x, u\tilde{\beta}) + 2 - 1 + 0 = -\ell(y, v\tilde{\beta}) - \ell(x, u\tilde{\beta}) + 1$. For length positivity to hold, we must therefore have at least one of $\ell(x, u\tilde{\beta}), \ell(y, v\tilde{\beta}) = 0$, and so either $us_{\tilde{\beta}} \in \LP(x)$ or $vs_{\tilde{\beta}} \in \LP(y)$, breaking distance-minimality, giving a contradiction.

    \vbls

    In the second case, we must have an even number of downwards edges (possibly zero) in a shortest path $u \Rightarrow w_yv$. If we do have downwards edges, then it can be seen from the diagram that $u, w_yv$ must end in the same simple reflection, and so $\Phi^+(u\tilde{\beta}) = \Phi^+(w_yv\tilde{\beta})$. For length positivity, we must therefore have that $\ell(x, u\tilde{\beta}) = \ell(y, v\tilde{\beta}) = 0$. Choosing $\tilde{\beta}$ to be simple, we find $us_{\tilde{\beta}} \in \LP(x), vs_{\tilde{\beta}} \in \LP(y)$, again breaking distance-minimality.

    Hence, we must have no downwards edges in a shortest path, i.e. $u \leq w_y v$. Again, choose $\tilde{\beta}$ to be simple. Then since $\langle \wt(u \Rightarrow w_yv), \tilde{\beta} \rangle = 0$, we must have either $us_{\tilde{\beta}} \in \LP(x)$ or $vs_{\tilde{\beta}} \in \LP(y)$ to ensure length positivity. If $us_{\tilde{\beta}} \in \LP(x)$, then we must have $us_{\tilde{\beta}} < u$ for distance-minimality, unless $u, w_yv$ differ by a single diagonal edge. This would then give that $\Phi^+(u\tilde{\beta}) = 0$, and so we also need $vs_{\tilde{\beta}} \in \LP(y)$. Distance-minimality then also implies that $\Phi^+(w_yv\tilde{\beta}) = 1$, forcing $r \notin \LP(x*y)$. The same occurs if we first assume that $vs_{\tilde{\beta}} \in \LP(y)$.

    \vbls
    
    Therefore, the only possible situation that can occur is that $u, w_yv$ differ by a single diagonal edge, giving $\Phi^+(u\tilde{\beta}) = \Phi^+(w_yv\tilde{\beta})$, and forcing $\ell(x, u\tilde{\beta}) = \ell(y, v\tilde{\beta}) = \ell(x*y, v\tilde{\beta}) = 0$. This also tells us that $\Inv(u) \subset \Inv(w_yv)$. We now check that in this case, $v \in \LP(x*y)$.
    
    \begin{align*}
        \ell(x*y, v\tilde{\alpha})\
        &=\ \ell(y, v\tilde{\alpha}) + \ell(x, u\tilde{\alpha}) + \Phi^+(w_yv\tilde{\alpha}) - \Phi^+(u\tilde{\alpha}) \\
        &<\ 0\ \iff\ \ell(x, u\tilde{\alpha}) = \ell(y, v\tilde{\alpha}) = 0,\ \Phi^+(w_yv\tilde{\alpha}) = 0,\ \Phi^+(u\tilde{\alpha}) = 1. 
    \end{align*}    

    This occurs when $\tilde{\alpha} \in \Inv(w_yv) / \Inv(u) = \{u^{-1}\tilde{\alpha}_i\}$ for some $i \in \{0, 1\}$ satisfying $w_y v= s_i u$, and when $us_{\tilde{\alpha}} \in \LP(x), vs_{\tilde{\alpha}} \in \LP(y)$. This implies that $\ell(x*y, v\tilde{\alpha}) < 0$ only when $\tilde{\alpha} = \tilde{\beta}$, since $ us_{\tilde{\beta}} \in \LP(x)$. However, $\ell(x*y, v\tilde{\beta}) = 0$ by the above, a contradiction. Hence $\ell(x*y, v\tilde{\alpha}) > 0$ for every positive $\tilde{\alpha}$ and so $v \in \LP(x*y)$.

    \vbls

    Since $\ell(x, u\tilde{\beta}) = \ell(y, v\tilde{\beta}) = \ell(x*y, v\tilde{\beta}) = 0$, by Prop. \ref{prop:lp-iff-zero}, we have that $\tilde{\beta}$ is simple $\implies us_{\tilde{\beta}} \in \LP(x) \iff vs_{\tilde{\beta}} \in \LP(y) \iff vs_{\tilde{\beta}} \in \LP(x*y)$. If $|M_{x,y}| = 2$, then there is one $\tilde{\beta}$, which must be simple, and so $M_{x,y}^y = \LP(x*y)$. If instead $|M_{x,y}| = 3$, then we can have chosen $v$ to be in between the other length positive elements for $y$ with respect to the Bruhat order, in which case there are two such $\tilde{\beta}$ which are both simple. Hence, in all cases, $M_{x,y}^y = \LP(x*y)$.

\end{proof}
\vbls
\section{Associativity}
\subsection{Notation and Setup}

We fix the notation $M_{x_1, x_2} = M_{1,2}$, and such that the Demazure product of two elements $x_1, x_2 \in \WTits$ with $(u_{1,2}, v_{1,2}) \in M_{1,2}$ is given as follows: 

\[ x_1 * x_2 = (w_1 \varepsilon^{\mu_1}) * (w_2 \varepsilon^{\mu_2}) = w_1 u_{1,2} v_{1,2}^{-1} \varepsilon^{v_{1,2} u_{1,2}^{-1} \mu_1 + \mu_2 - v_{1,2} \wt(u_{1,2} \Rightarrow w_2 v_{1,2})}. \]

\vspace{0.5\baselineskip}

We also use the notation $x_i * x_j = x_{ij} = w_{ij}\varepsilon^{\mu_{ij}}$ throughout, extending this to other objects such as $(u_{ij,k}, v_{ij,k}) \in M_{ij,k} \subseteq \LP(x_{ij}) \times \LP(x_k)$. The triple products $(x_1 * x_2) * x_3$ and $x_1 * (x_2 * x_3)$ are given in this notation below.

\vspace{-0.5\baselineskip}

\begin{align*}
    (x_1 * x_2) * x_3 &= w_1 u_{1,2} v_{1,2}^{-1} u_{12,3} v_{12,3}^{-1} \varepsilon^{v_{12,3}u_{12,3}^{-1}(v_{1,2}u_{1,2}^{-1}\mu_1 + \mu_2 - v_{1,2}\wt(u_{1,2} \Rightarrow w_2 v_{1,2})) + \mu_3 - v_{12,3}\wt(u_{12,3} \Rightarrow w_3 v_{12,3})},\\
    x_1 * (x_2 * x_3) &=  w_1 u_{1,23} v_{1,23}^{-1} \varepsilon^{v_{1,23}u_{1,23}^{-1}\mu_1 + v_{2,3}u_{2,3}^{-1}\mu_2 + \mu_3 - v_{2,3}\wt(u_{2,3} \Rightarrow w_3 v_{2,3}) - v_{1,23}\wt(u_{1,23} \Rightarrow w_2 u_{2,3} v_{2,3}^{-1} v_{1,23})}.
\end{align*}

In the rest of this section, we aim to show that, for some suitable choices of $(u_{i,j}, v_{i,j}) \in M_{i,j}$, these formulae are equal, and are therefore always equal by well-definedness. We fix that $\lev(x), \lev(y) > 1$, and so $|\LP(x_i)| \leq 2$, to ease calculations.

We know that each element $x \in \WTits$ can be written uniquely as $x = w\varepsilon^{\mu}$ for some $w \in W, \mu \in \mathcal{T}$. Hence, for the left and right triple products to be equal, we must have that the Weyl components $w_{12,3}, w_{1,23}$ and the coweight lattice components $\mu_{12,3}, \mu_{1,23}$ are each equal. 

\subsection{Weyl Component}

We can simplify the condition for the Weyl components to be equal as follows

\vspace{-0.5\baselineskip}

\begin{align*}
    w_1 u_{1,2} v_{1,2}^{-1} u_{12,3} v_{12,3}^{-1} = w_1 u_{1,23} v_{1,23}^{-1}\
    &\iff\ u_{1,2} v_{1,2}^{-1} u_{12,3} v_{12,3}^{-1} = u_{1,23} v_{1,23}^{-1}\\
    &\iff\ g := u_{1,23}^{-1} u_{1,2} v_{1,2}^{-1} u_{12,3} v_{12,3}^{-1} v_{1,23}\ =\ e.
\end{align*}

Our goal is then to show that $g = e$. We first simplify $g$ before splitting into cases.

\begin{prop}
    There is some choice of distance minimising pairs such that $v_{1,2} = u_{12,3}$ and $v_{2,3} = v_{1,23}$.
\end{prop}

\begin{proof}
Note that $u_{12,3} \in \LP(x_{12})$, and so there is some pair $(-, u_{12,3}) \in M_{1,2}$. However, we may choose any pair $(u_{1,2}, v_{1,2}) \in M_{1,2}$ for the calculation of the product $x_1 * x_2$, and so we are free to fix $v_{1,2}$ by fixing the pairs $(u_{1,2}, v_{1,2}) = (-, u_{12,3})$, giving us the choice $v_{1,2} = u_{12,3}$.

Similarly, $v_{1,23} \in \LP(x_{23})$, and so we can fix the pairs $(u_{2,3}, v_{2,3}) = (-, v_{1,23})$, giving us a choice of $v_{2,3} = v_{1,23}$ as required.
\end{proof}

With these choices, we obtain the following simplified expression for $g$.

\[ g = u_{1,23}^{-1} u_{1,2} v_{1,2}^{-1} u_{12,3} v_{12,3}^{-1} v_{1,23} = u_{1,23}^{-1} u_{1,2} u_{12,3}^{-1} u_{12,3} v_{12,3}^{-1} v_{1,23} = u_{1,23}^{-1} u_{1,2} v_{12,3}^{-1} v_{1,23}. \]

\vspace{0.5\baselineskip}

Next, we investigate when the element $u_{1,23}^{-1} u_{1,2} = e$. To do this, we use the following statement.

\begin{prop}
    If $u_{12,3} = u_{2,3}$, then both $v_{12,3} = v_{1,23}$ and $u_{1,23} = u_{1,2}$.
\end{prop}

\begin{proof}
Thm. \ref{thm:vertical-edge} immediately gives us the first part of the claim by noting that $u_{12,3}, u_{2,3} \in \LP(x_2)$.  For the second part, note that if $w_2 v_{1,2} = w_{23}v_{1,23}$, then $u_{1,2} = u_{1,23}$, by Thm. \ref{thm:vertical-edge}, and so we consider the following calculation:

\vspace{-\baselineskip}

\begin{align*}
    (w_2 v_{1,2})^{-1}(w_{23}v_{1,23}) \
    =\ v_{1,2}^{-1} w_2^{-1} w_2 u_{2,3} v_{2,3}^{-1} v_{1,23}\
    =\ v_{1,2}^{-1} u_{2,3}\
    =\ u_{12,3}^{-1} u_{2,3}.
\end{align*}

Hence, if $u_{12,3} = u_{2,3}$, then $(w_2 v_{1,2})^{-1}(w_{23}v_{1,23}) = u_{12,3}^{-1} u_{2,3} = e$, giving us that $w_2 v_{1,2} = w_{23}v_{1,23}$ and so $u_{1,2} = u_{1,23}$, proving the rest of the claim.
\end{proof}

As an immediate consequence, we get that if $u_{12,3} = u_{2,3}$, then $g = e$ and we are done, and so henceforth we will assume that we cannot have $u_{12,3} = u_{2,3}$, so in particular $|\LP(x_2)| = 2$. Note also that everything we have stated so far applies to any double affine Weyl semigroup with a well-defined Demazure product, however for the rest of this proof we will work explicitly with the case of $\widehat{SL}_2$. 

By Thm. \ref{thm:vertical-edge}, $\LP(x_2)$ is connected vertically in $\QBG(W)$. We may therefore write $u_{12,3} = u_{2,3}s_i$ for some $i \in \{0,1\}$, which gives us $w_{23}v_{1,23} = w_2v_{1,2}s_i$.

\vspace{\baselineskip}

We split the remaining work into two cases; either $u_{1,23} \neq u_{1,2}$, or $u_{1,23} = u_{1,2}$.

In the first instance, note that $\LP(x_2) = \{v_{1,2}, v_{1,2}s_i\}$ and $\LP(x_1) = \{u_{1,2}, u_{1,23}\}$. Hence, for distance minimality to hold, we must have a situation as in Fig. \ref{fig:qbg-portion}, otherwise we would require $u_{1,2} = u_{1,23}$. Therefore $|M_{1,2}| = 2$, and from the diagram we find that $u_{1,2} = u_{1,23}s_i$.Furthermore, by assumption we cannot choose $u_{12,3} = u_{2,3}$ despite having two options for $u_{12,3}$, and so again we must have an alignment as in Fig. \ref{fig:qbg-portion}. Therefore $v_{12,3} = v_{2,3}s_i$, and so $g = u_{1,23}^{-1}u_{1,2}v_{12,3}^{-1} = s_i s_i = e$.

In the second instance, we have $u_{1,23} = u_{1,2}$. If we also had $|M_{2,3}| = 2$, then $\LP(x_3) = \{v_{2,3}, v_{2,3}s_i\}$, and so we would have the freedom to choose $v_{1,23} = v_{2,3}$. For distance minimality we would therefore need $w_{23}v_{1,23} =w_2v_{1,2}$, a contradiction. Hence $|M_{2,3}| = 1$ and so $v_{2,3} = v_{12,3}$ is forced, giving $g = e$ as required. We have therefore proved the following:

\begin{prop}
    The generalised Demazure product associated to affine $\widehat{SL}_2$ is associative in the Weyl component for level greater than one.
\end{prop}

\subsection{Lattice Component}

We can write down the conditions in the coweight lattice part for associativity directly from the left and right triple products that were calculated earlier. In this section, we continue to assume that $\lev(x), \lev(y) > 1$.

\vspace{-0.5\baselineskip}

\begin{align*}
    &v_{12,3}u_{12,3}^{-1}(v_{1,2}u_{1,2}^{-1}\mu_1 + \mu_2 - v_{1,2}\wt(u_{1,2} \Rightarrow w_2 v_{1,2})) + \mu_3 - v_{12,3}\wt(u_{12,3} \Rightarrow w_3 v_{12,3}) \\
    =\  &v_{1,23}u_{1,23}^{-1}\mu_1 + v_{2,3}u_{2,3}^{-1}\mu_2 + \mu_3 - v_{2,3}\wt(u_{2,3} \Rightarrow w_3 v_{2,3}) - v_{1,23}\wt(u_{1,23} \Rightarrow w_2 u_{2,3} v_{2,3}^{-1} v_{1,23}).
\end{align*}

\vspace{0.5\baselineskip}

We can immediately reduce this expression; the $\mu_1$ components cancel out, as the corresponding equations are simply associativity in the Weyl part, and the $\mu_3$ components cancel trivially. So, as long as we make the same choices as we did previously, namely that $v_{1,2} = u_{12,3}$ and $v_{2,3} = v_{1,23}$, we obtain the following condition:

\vspace{-0.5\baselineskip}

\begin{align*}
    &v_{12,3}u_{12,3}^{-1}\mu_2 - v_{12,3}u_{12,3}^{-1}v_{1,2}\wt(u_{1,2} \Rightarrow w_2 v_{1,2}) - v_{12,3}\wt(u_{12,3} \Rightarrow w_3 v_{12,3}) \\
    =\  &v_{2,3}u_{2,3}^{-1}\mu_2 - v_{2,3}\wt(u_{2,3} \Rightarrow w_3 v_{2,3}) - v_{1,23}\wt(u_{1,23} \Rightarrow w_2 u_{2,3} v_{2,3}^{-1} v_{1,23}).
\end{align*}

\vspace{0.5\baselineskip}

We can then apply our choices and group together different elements to further simplify this expression as much as we can before investigating individual components.

\vspace{-0.5\baselineskip}

\begin{align*}
    (v_{12,3}u_{12,3}^{-1}\mu_2 - v_{2,3}u_{2,3}^{-1}\mu_2)
    &- (v_{12,3}\wt(u_{1,2} \Rightarrow w_2 v_{1,2}) - v_{1,23}\wt(u_{1,23} \Rightarrow w_2 u_{2,3})) \\
    &- (v_{12,3}\wt(u_{12,3} \Rightarrow w_3 v_{12,3}) - v_{1,23}\wt(u_{2,3} \Rightarrow w_3 v_{2,3}))\
    =\ 0.
\end{align*}

Looking at this, we define the following three objects $\eta_1, \eta_2, \eta_3 \in P^{\vee}$ such that our conditions reduce to showing that $\eta_1 - \eta_2 - \eta_3 = 0$.

\vspace{-0.5\baselineskip}

\begin{align*}
    \eta_1 &:= v_{12,3}u_{12,3}^{-1}\mu_2 - v_{2,3}u_{2,3}^{-1}\mu_2, \\
    \eta_2 &:= v_{12,3}\wt(u_{1,2} \Rightarrow w_2 v_{1,2}) - v_{1,23}\wt(u_{1,23} \Rightarrow w_2 u_{2,3}), \\
    \eta_3 &:= v_{12,3}\wt(u_{12,3} \Rightarrow w_3 v_{12,3}) - v_{1,23}\wt(u_{2,3} \Rightarrow w_3 v_{2,3}).
\end{align*}

We now begin to investigate what values the elements $\eta_1, \eta_2, \eta_3$ can take, and what conditions need to be met for these values to be obtained.

\begin{prop}
    If the following three conditions are met, then $\eta_2 = -v_{2,3}(\tilde{\alpha}_i^{\vee})$, otherwise $\eta_2 = 0$.
    \begin{enumerate}[label=\normalfont(\roman*)]
        \item $u_{2,3} = v_{1,2}s_i$.
        \item $v_{2,3} = v_{12,3}$.
        \item $\ell(u_{1,2}) \geq \ell(w_2 v_{1,2}) > \ell(w_2 v_{1,2} s_i)$.
    \end{enumerate}
\end{prop}

\begin{proof}
    First, assume condition (i) cannot hold. Since $|\LP(x)| \leq 2$, we must have that $u_{2,3} = v_{1,2}$, giving us the following.
    \[ w_{23}v_{1,23} = w_2 u_{2,3} v_{2,3}^{-1} v_{1,23} = w_2 u_{2,3} = w_2 v_{1,2} \]
    Where we have used the choice that $v_{2,3} = v_{1,23}$. However, uniqueness of distance minimising elements then forces $u_{1,2} = u_{1,23}$. We also have $u_{2,3} = v_{1,2} = u_{12,3}$ from our previous choice, and so $v_{2,3} = v_{12,3}$ by uniqueness in the other direction. We then have:

    \vbls[-1]
    
    \begin{align*}
        \eta_2\ &=\ v_{12,3}\wt(u_{1,2} \Rightarrow w_2 v_{1,2}) - v_{1,23}\wt(u_{1,23} \Rightarrow w_2 u_{2,3}) \\
        &=\ v_{2,3}\wt(u_{1,2} \Rightarrow w_2 v_{1,2}) - v_{2,3}\wt(u_{1,2} \Rightarrow w_2 v_{1,2})\ =\ 0.
    \end{align*}

    Now we assume that condition (i) holds, and that $u_{2,3} = v_{1,2}s_i$ for some $i \in \{0,1\}$. It then immediately follows that $|\LP(x_2)| = 2$, with $w_{23} v_{1,23} = v_2 v_{1,2} s_i$ following an earlier calculation, and $u_{2,3} = u_{12,3} s_i$ from our choices. We can then split this into 4 different cases depending on the positioning of $u_{1,2}, u_{1,23}, w_2v_{1,2}, w_{23}v_{1,23}$ within $\QBG(W)$, up to the involution $s_0 \leftrightarrow s_1$ which corresponds to mirroring the diagrams along the vertical centre.

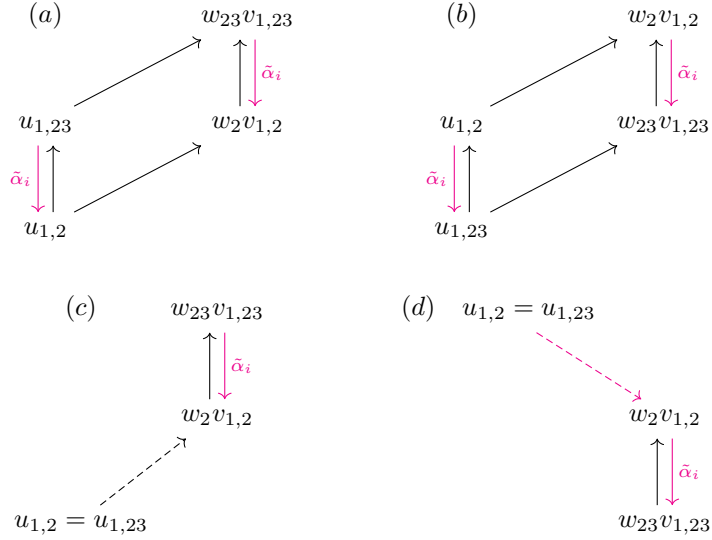
\begin{figure}[ht]
\begin{center}
\begin{tikzcd}[row sep=0.9cm, column sep=0.7cm]
(a) & & w_{23} v_{1,23}
\arrow[d, magenta, shift left = 0.1cm, "\tilde{\alpha}_i"] \\
u_{1,23}
\arrow[urr]
\arrow[d, magenta, shift right = 0.1cm, "\tilde{\alpha}_i"']
& & w_{2} v_{1,2}
\arrow[u, shift left = 0.1cm] \\
u_{1,2}
\arrow[u, shift right = 0.1cm] 
\arrow[urr]
\end{tikzcd}
\qquad\qquad
\begin{tikzcd}[row sep=0.9cm, column sep=0.7cm]
(b) & & w_2 v_{1,2}
\arrow[d, magenta, shift left = 0.1cm, "\tilde{\alpha}_i"] \\
u_{1,2}
\arrow[urr]
\arrow[d, magenta, shift right = 0.1cm, "\tilde{\alpha}_i"']
& & w_{23} v_{1,23}
\arrow[u, shift left = 0.1cm] \\
u_{1,23}
\arrow[u, shift right = 0.1cm] 
\arrow[urr]
\end{tikzcd}
\\ \vspace{\baselineskip}
\begin{tikzcd}[row sep=0.9cm, column sep=0cm]
(c) & w_{23} v_{1,23}
\arrow[d, magenta, shift left = 0.1cm, "\tilde{\alpha}_i"] \\
& w_{2} v_{1,2} 
\arrow[u, shift left = 0.1cm] \\
u_{1,2} = u_{1,23}
\arrow[ur, dashed]
\end{tikzcd}
\qquad\qquad
\begin{tikzcd}[row sep=0.9cm, column sep=0cm]
(d) \quad u_{1,2} = u_{1,23}
\arrow[dr, magenta, dashed] \\
& w_{2} v_{1,2}
\arrow[d, magenta, shift left = 0.1cm, "\tilde{\alpha}_i"] \\
& w_{23} v_{1,23}
\arrow[u, shift left = 0.1cm]
\end{tikzcd}
\caption{Portions of $\QBG(W)$ to be considered in associativity.}
\end{center}
\end{figure}

This exhausts all possible cases, as any other arrangement will violate either uniqueness of distance-minimising elements or the minimality itself. 

In cases (a), (b), and (c), the weights that $\eta_2$ depends on are both 0 and so $\eta_2 = 0$. Hence, the only case left to check is (d), which occurs precisely when $\ell(u_{1,2}) \geq \ell(w_2 v_{1,2}) > \ell(w_2 v_{1,2} s_i)$, i.e. when condition (iii) holds.

\vspace{\baselineskip}

We can now further split this into 2 more cases depending on whether or not condition (ii) holds. First, assume that it does not hold, and so $v_{2,3} = v_{12,3}s_i$. Then we must have that $|M_{2,3}| = 2$, as we have two different distance minimising paths from $\LP(x_2)$ to $w_3\LP(x_3)$, namely $u_{2,3} \Rightarrow w_3 v_{2,3}$ and $u_{12,3} \Rightarrow w_3 v_{12,3}$. However, this means that the element $v_{1,23}$ is unrestricted, which contradicts diagram (d), as $w_{23} v_{12,3} s_i = w_{23} v_{2,3}$ would be closer to $u_{1,23}$ than $v_{1,23}$. Hence, condition (ii) must hold and so $v_{2,3} = v_{12,3}$. We are then left to calculate $\eta_2$ directly. First, note that $\wt(u_{1,23} \Rightarrow w_{23} v_{1,23}) = \wt(u_{1,2} \Rightarrow w_2 v_{1,2}) + \tilde{\alpha}^{\vee}_i$, as indicated from diagram (d). Hence:

\vspace{-\baselineskip}

\begin{align*}
    \eta_2
    &= v_{12,3}\wt(u_{1,2} \Rightarrow w_2 v_{1,2}) - v_{1,23}\wt(u_{1,23} \Rightarrow w_2 u_{2,3}) \\
    &= v_{2,3}\wt(u_{1,2} \Rightarrow w_2 v_{1,2}) - v_{2,3}(\wt(u_{1,2} \Rightarrow w_2 v_{1,2}) + \tilde{\alpha}^{\vee}_i)
    = - v_{2,3}(\tilde{\alpha}^{\vee}_i).
\end{align*}

\vspace{-\baselineskip}

\end{proof}

Now that we have proved claim 3.1, we next want to investigate $\eta_3$. The following proposition gives the following values. The proof is analogous to the case of $\eta_2$.

\begin{prop}
    If the following three conditions are met, then $\eta_3 = v_{2,3}(\tilde{\alpha}_i^{\vee})$, otherwise $\eta_3 = 0$.
    \begin{enumerate}[label=\normalfont(\roman*)]
        \item $u_{2,3} = v_{1,2}s_i$.
        \item $v_{2,3} = v_{12,3}$.
        \item $\ell(u_{2,3}s_i) > \ell(u_{2,3}) \geq \ell(w_3 v_{2,3})$.
    \end{enumerate}
\end{prop}

We now have complete descriptions of the behaviour of $\eta_2, \eta_3$, so now we want to consider the value of $\eta_1$ under the same conditions in order to show that $\eta_1 = \eta_2 + \eta_3$. Firstly, we list the different cases that can appear, along with the value of $\eta_2 + \eta_3$ in each case.

\begin{enumerate}[label=(\roman*)]
    \item $u_{2,3} = u_{12,3},\ \implies\ \eta_2 + \eta_3 = 0$.
    \item $u_{2,3} = u_{12,3} s_i,\ v_{2,3} = v_{12,3} s_i,\ \implies\ \eta_2 + \eta_3 = 0$.
    \item $u_{2,3} = u_{12,3} s_i,\ v_{2,3} = v_{12,3}$, and:
    \begin{enumerate}[label=(\alph*)]
        \item $\ell(u_{1,2}) < \ell(w_2 v_{1,2}) < \ell(w_2 v_{1,2} s_i),\ \ell(u_{2,3}s_i) < \ell(u_{2,3}) < \ell(w_3 v_{2,3}),\ \implies\ \eta_2 + \eta_3 = 0$.
        \item $\ell(u_{1,2}) \geq \ell(w_2 v_{1,2}) > \ell(w_2 v_{1,2} s_i),\ \ell(u_{2,3}s_i) > \ell(u_{2,3}) \geq \ell(w_3 v_{2,3}),\ \implies\ \eta_2 + \eta_3 = 0$.
        \item $\ell(u_{1,2}) \geq \ell(w_2 v_{1,2}) > \ell(w_2 v_{1,2} s_i),\ \ell(u_{2,3}s_i) < \ell(u_{2,3}) < \ell(w_3 v_{2,3}),\ \implies\ \eta_2 + \eta_3 = -v_{2,3}(\tilde{\alpha}^{\vee}_i)$.
        \item $\ell(u_{1,2}) < \ell(w_2 v_{1,2}) < \ell(w_2 v_{1,2} s_i),\ \ell(u_{2,3}s_i) > \ell(u_{2,3}) \geq \ell(w_3 v_{2,3}),\ \implies\ \eta_2 + \eta_3 = v_{2,3}(\tilde{\alpha}^{\vee}_i)$.
    \end{enumerate}
\end{enumerate}

Note that this exhausts all cases, as we must have either $\ell(u_{1,2}) < \ell(w_2 v_{1,2}) < \ell(w_2 v_{1,2} s_i)$ or $\ell(u_{1,2}) \geq \ell(w_2 v_{1,2}) > \ell(w_2 v_{1,2} s_i)$, by uniqueness of distance minimising elements in $\QBG(W)$. Finally, we are left to calculate the value of $\eta_1$ in each case, and show that we must have $\eta_1 = \eta_2 + \eta_3$.

\vspace{\baselineskip}

\textbf{Case (i)}\quad We have that $u_{2,3} = u_{12,3}$ and so $v_{2,3} = v_{12,3}$ by uniqueness, which immediately gives: \[ \eta_1 = v_{12,3}u_{12,3}^{-1}\mu_2 - v_{2,3}u_{2,3}^{-1}\mu_2 = v_{2,3}u_{2,3}^{-1}\mu_2 - v_{2,3}u_{2,3}^{-1}\mu_2 = 0.\]

\textbf{Case (ii)}\quad We have $u_{2,3} = u_{12,3}s_i$ and $v_{2,3} = v_{12,3} s_i$, and again we immediately find: \[\eta_1 = v_{12,3}u_{12,3}^{-1}\mu_2 - v_{2,3}u_{2,3}^{-1}\mu_2 = v_{2,3} s_i s_i u_{2,3}^{-1}\mu_2 - v_{2,3} u_{2,3}^{-1}\mu_2 = 0.\]

\textbf{Case (iii)}\quad Before considering the four subsequent cases separately, let us first simplify $\eta_1$.

\vbls[-1]

\begin{align*}
    \eta_1 
    = v_{12,3}u_{12,3}^{-1}\mu_2 - v_{2,3}u_{2,3}^{-1}\mu_2 
    &= v_{2,3}(s_i u_{2,3}^{-1} \mu_2 - u_{2,3}^{-1} \mu_2) \\
    &= v_{2,3}(u_{2,3}^{-1}\mu_2 - \langle u_{2,3}^{-1}\mu_2, \tilde{\alpha}_i \rangle \tilde{\alpha}_i^{\vee} - u_{2,3}^{-1}\mu_2) \\
    &= -\langle \mu_2, u_{2,3}\tilde{\alpha}_i\rangle v_{2,3}(\tilde{\alpha}_i^{\vee}). \numberthis\label{eqn:eta1_value}
\end{align*}

And so we seek to compute the value of $\langle \mu_2, u_{2,3}\tilde{\alpha}_i \rangle$ in each sub-case. We do this using the classification of length positive elements to explicitly find the possible values of $\mu_2$ and $u_{2,3}$, showing that the bilinear product takes the values 0, 1, or -1 in the relevant cases labelled above.

\vbls

In all sub-cases, we have $|\LP(x_2)| = 2$. Write $\mu_2 = k\alpha^{\vee} + m\delta + l\Lambda_0$, and fix $t \in \mathbb{Z}$ such that $-\frac{l}{2} < j = k - tl \leq \frac{l}{2}$. From the proof of Corollary \ref{cor:LPx-size}, we must have that $j \in \{0, \frac{l}{2}, \pm\frac{l-1}{2}\}$. By Prop. \ref{prop:LP-algorithm}, the possible values of pairs $(j, u)$ for $u \in \LP(x_2) = \{u_{2,3}, v_{1,2}\}$ are:

\begin{alignat*}{4}
    &(0, \tau^{t\alpha^\vee}), &&(\frac{l}{2}, \tau^{t\alpha^\vee}), &&(\frac{l-1}{2}, \tau^{t\alpha^\vee}), &&(-\frac{l-1}{2}, \tau^{(t-1)\alpha^\vee}), \\
    &(0, \tau^{t\alpha^\vee}s_1),\ &&(\frac{l}{2}, \tau^{t\alpha^\vee}s_0),\ &&(\frac{l-1}{2}, \tau^{t\alpha^\vee}s_0),\ &&(-\frac{l-1}{2}, \tau^{(t-1)\alpha^\vee}s_0).
\end{alignat*} 


Write $u_{2,3} = v_0\tau^{t'\alpha^\vee}$. Taking $i = 0$ and $i=1$ respectively, we find:

\vbls[-1]

\begin{align*}
    \langle \mu_2, u_{2,3}\tilde{\alpha}_1 \rangle 
    &= \langle k\alpha^{\vee} + m\delta + l\Lambda_0, v_0\tau^{t'\alpha^\vee}(\alpha) \rangle \\
    &= \langle k\alpha^{\vee} + m\delta + l\Lambda_0, v_0\alpha + 2t'\delta \rangle\ 
    = 2(k\sgn(v_0) - lt'). \\[0.5\baselineskip]
    \langle \mu_2, u_{2,3}\tilde{\alpha}_0 \rangle 
    &= \langle k\alpha^{\vee} + m\delta + l\Lambda_0, v_0\tau^{t'\alpha^\vee}(-\alpha + \delta) \rangle \\
    &= \langle k\alpha^{\vee} + m\delta + l\Lambda_0, -v_0\alpha + (2t'+1)\delta \rangle\ 
    = -2(k\sgn(v_0) - lt') + l.
\end{align*}

Before investigating the sub-cases, we calculate these products for each possible pair $(j,u)$, and then investigate which pairs are valid for which sub-cases. It is clear from the relevant pairs which value of $i$ is required, as $j$ takes the same value for a fixed $x \in \WTits$, and hence is the same for $u_{2,3}, v_{1,2}$, where $u_{2,3} = v_{1,2}s_i$. We label the pairs in the form $(j,u_0)$, where $u_0$ is the finite Weyl component of $u$. 

\vbls

\begin{center}
\def\arraystretch{1.2}
\begin{tabular}{ c | c c c c c c c c }
    $(j,u_0)$ & $(0, e)$ & $(0, s_1)$ & $(\frac{l}{2}, e)$ & $(\frac{l}{2}, s_1)$ & $(\frac{l-1}{2}, e)$ & $(\frac{l-1}{2}, s_1)$ & $(-\frac{l-1}{2}, e)$ & $(-\frac{l-1}{2}, s_1)$ \\
    \hline 
    $\langle \mu_2, u_{2,3}\tilde{\alpha}_i\rangle$ & 0 & 0 & 0 & 0 & 1 & -1 & -1 & 1  
\end{tabular}
\end{center}

\vbls[0.5]

For the following, we write $w_2 = w_0\tau^{r\alpha^{\vee}}$, where $w_0 \in \Wfin$ and $r \in \mathbb{Z}$. 

\vbls

\textbf{Case (a)} $v_{1,2} < u_{2,3}$ and $w_2 v_{1,2} < w_2 u_{2,3}$. If $j = \frac{l-1}{2}$, then $u_{2,3}, v_{1,2} \in \{\tau^{t\alpha^{\vee}}, s_1\tau^{-(t+1)\alpha^{\vee}}\}$. Note that from Prop. \ref{prop:LP-algorithm}, the latter of these being length positive requires $-r \leq t \leq -1$, forcing $r + t \geq 0$. If $u_{2,3} = s_1\tau^{-(t+1)\alpha^{\vee}}$, then for $u_{2,3} > v_{1,2}$ we must have $t \geq 0$, a contradiction. If instead $u_{2,3} = \tau^{t\alpha^{\vee}}$, then $w_2u_{2,3} = w_0\tau^{r\alpha^{\vee}} \tau^{t\alpha^{\vee}} = w_0 \tau^{(r+t)\alpha^{\vee}} = w_0(s_0s_1)^{r+t}$. However, $w_2v_{1,2} = w_2u_{2,3}s_0 = w_0(s_0s_1)^{r+t}s_0 > w_2u_{2,3}$ as $r+t \geq 0$, contradicting the case. If $j = -\frac{l-1}{2}$, an analogous argument leads to the same contradictions. Hence $j = 0$ or $j = \frac{l}{2}$, and in either case $\langle \mu_2, u_{2,3}\tilde{\alpha}_i \rangle = 0$ as required.

\vbls

\textbf{Case (b)} $v_{1,2} > u_{2,3}$ and $w_2 v_{1,2} > w_2 u_{2,3}$. We mirror the argument in case (a): If $j = \frac{l-1}{2}$, then as before we require $t \leq -1$ and $r+t \geq 0$. If $u_{2,3} = \tau^{t\alpha^{\vee}}$, then for $u_{2,3} < v_{1,2}$ we require $t \geq 0$, a contradiction. If instead $u_{2,3} = s_1\tau^{-(t+1)\alpha^{\vee}}$, then $w_2 u_{2,3} = w_0 s_1 \tau^{-1-r-t} = w_0s_0 (s_1s_0)^{r+t}$, and $w_2v_{1,2} = w_0s_0(s_1s_0)^{r+t}s_0 < w_2u_{2,3}$ since $r+t \geq 0$, another contradiction. Repeating the argument with $j = -\frac{l-1}{2}$ gives the same result, and so as before we must have $j = 0$ or $j = \frac{l}{2}$, and so $\langle \mu_2, u_{2,3}\tilde{\alpha}_i \rangle = 0$.

\vbls

\textbf{Case (c)} $v_{1,2} < u_{2,3}$ and $w_2 v_{1,2} > w_2 u_{2,3}$. 

If $j = 0$ then $u_{2,3}, v_{1,2}$ take the values $\tau^{t\alpha^{\vee}}, s_1\tau^{-t\alpha^{\vee}}$. If $u_{2,3} = \tau^{t\alpha^{\vee}}$, then we must have $t \geq 1$ to ensure $v_{1,2} < u_{2,3}$. Hence by Prop. \ref{prop:LP-algorithm}, we must have $t \geq \Phi^+(w_0\alpha) - r$. However, $w_2 v_{1,2} > w_2 u_{2,3} \implies w_0s_1\tau^{-(t+r)\alpha^{\vee}} > w_0\tau^{(t+r)\alpha^{\vee}} \implies t + r \leq \Phi^+(w_0\alpha) - 1$, a contradiction. Similarly, if $u_{2,3} = s_1\tau^{-t\alpha^{\vee}}$, we require $t \leq 0$ for $v_{1,2} < u_{2,3}$ and $t \leq \Phi^+(w_0\alpha) - r - 1$ for length positivity. However, $w_2 v_{1,2} > w_2u_{2,3} \implies t+r \geq \Phi^+(w_0\alpha)$, a contradiction, and so $j \neq 0$. If $j = \frac{l}{2}$, an analogous argument leads to the same contradictions, and so $j = \pm\frac{l-1}{2}$.

If $j = \frac{l-1}{2}$ and $u_{2,3} = \tau^{t\alpha^{\vee}}s_0$, then we require $t \geq 0$ to ensure $u_{2,3} > v_{1,2}$, and we require $-r \leq t \leq -1$ for length positivity, a contradiction. If $j = -\frac{l-1}{2}$ and $u_{2,3} = \tau^{(t-1)\alpha^{\vee}}$, then we require $t \leq 0$ to ensure $u_{2,3} > v_{1,2}$, and $1 \leq t \leq -r$ for length positivity, another contradiction.

In all remaining possibilities, $\langle \mu_2, u_{2,3}\tilde{\alpha}_i \rangle = 1$, and so $\eta_1 = -v_{2,3}(\tilde{\alpha}^{\vee}_i)$ by equation \ref{eqn:eta1_value}, as required.

\vbls

\textbf{Case (d)} $v_{1,2} > u_{2,3}$ and $w_2 v_{1,2} < w_2 u_{2,3}$. We mirror the argument in case (c):

If $j = 0$, then $u_{2,3}, v_{1,2}$ take the values $\tau^{t\alpha^{\vee}}, s_1\tau^{-t\alpha^{\vee}}$. If $u_{2,3} = \tau^{t\alpha^{\vee}}$, then we need $t \leq 0$ to ensure $v_{1,2} > u_{2,3}$, and therefore we need $t \leq \Phi^+(w_0\alpha) - r - 1$ for length positivity. For $w_2v_{1,2} < w_2u_{2,3}$, we need $t+r \geq \Phi^+(w_0\alpha)$, a contradiction. If instead $u_{2,3} = s_1\tau^{-t\alpha^{\vee}}$, then we require $t \geq 1$ and $t \geq \Phi^+(w_0\alpha) - r$, but for $w_2v_{1,2} < w_2u_{2,3}$, we need $t + r \leq \Phi^+(w_0\alpha) - 1$, a contradiction. The same occurs if $j = \frac{l}{2}$, and so $j = \pm\frac{l-1}{2}$.

If $j = \frac{l-1}{2}$ and $u_{2,3} = \tau^{t\alpha^{\vee}}$, then we require $t \leq -1$ and $1 \leq t \leq -r$, a contradiction. If $j = -\frac{l-1}{2}$ and $u_{2,3} = \tau^{(t-1)\alpha^{\vee}}s_0$, then we require $t \geq 1$ and $-r \leq t \leq -1$, another contradiction. 

Similarly to the previous case, in all remaining possibilities we must have $\langle \mu_2, u_{2,3}\tilde{\alpha}_i \rangle = -1$, and so $\eta_1 = v_{2,3}(\tilde{\alpha}^{\vee}_i)$ as required. 

This completes all possible cases, and so $\eta_1 = \eta_2 + \eta_3$. Hence, we have shown that the lattice component of the generalised Demazure product is associative, and so we have proven the following.

\begin{thm}\label{thm:assoc}
    The generalised Demazure product for $\WTits$ of type $\widehat{SL}_2$ is associative for elements with level greater than one.
\end{thm}

\begin{rmk}
    In the case of $\lev(x) = 1$, the above arguments fail when $|\LP(x)| = 3$, for which many more cases appear from the distance minimality conditions. In this edge case, we have failed to find any non-associative examples, and we conjecture that associativity still holds, but the calculations become increasingly arduous. However, an analogous argument does hold for those level one elements whose length positive set is at most two elements. 
\end{rmk}
\vbls
\section{Length Additivity in General Type}
\subsection{Theorem Statement and Propositions.}

The goal of this section is to investigate and ultimately prove the following theorem.

\begin{thm}\label{bigone}
    Let $x, y \in W_{\mathcal{T}}$, and assume that $*_{u,v}^p = *$ is well-defined. Then $\ell(x * y) = \ell(x) + \ell(y) \iff \ell(xy) = \ell(x) + \ell(y)$. In this case, $x*y = xy$.
\end{thm}

\vbls[0.5]

We achieve this by proving an extension of work by Schremmer to the double affine case.

\vbls[0.5]

\begin{thm}\label{demlen}
    Assume that $*_{u,v}^p = *$ is well-defined, and let $(u,v) \in M_{x,y}$. Then:
    \[ \ell(x * y) = \ell(x) + \ell(y) - d(u \Rightarrow w_y v). \]
\end{thm}

\vbls[0.5]

To prove this, we need two extend two more propositions by Schremmer to the double-affine case. The first of these can be proved almost immediately by mimicking Schremmer's proof in the single-affine case.

It is important to note that, by assuming that $*$ is well-defined, we are also assuming that the weight function on $\QBG(W)$ is well-defined, and so the weight of two shortest paths is the same. This has yet to be proved for general type, although calculations suggest that this is reasonable to assume.

\vbls[0.5]

\begin{prop}
    Let $w_1, w_2 \in W$, and assume that the weight function for $\QBG(W)$ is well-defined. Then:
    \[ \langle \wt(w_1 \Rightarrow w_2), 2\rho \rangle = d(w_1 \Rightarrow w_2) + \ell(w_1) - \ell(w_2). \]
\end{prop}

\begin{proof}
    First, note that if $w_2 = w_1$, all factors are 0 and the claim holds. Now assume that $w_2 = w_1 s_{\tilde{\alpha}}$ such that $w_1 \rightarrow w_2$ is an edge in $\QBG(W)$, leading to one of two cases.
    \begin{enumerate}[label=(\roman*)]
        \item $w_1 \rightarrow w_1 s_{\tilde{\alpha}}$ is a Bruhat edge: Then $\ell(w_2) = \ell(w_1) + 1$ and $\wt(w_1 \Rightarrow w_2) = 0$.
        \item $w_1 \rightarrow w_1 s_{\tilde{\alpha}}$ is a quantum edge: Then $\ell(w_2) = \ell(w_1) + 1 - \langle \tilde{\alpha}^{\vee}, 2\rho \rangle$ and $\wt(w_1 \Rightarrow w_2) = \tilde{\alpha}^{\vee}$.
    \end{enumerate}
    In either case, we have that $\ell(w_1s_{\tilde{\alpha}}) = \ell(w_1) + 1 - \langle \wt(w_1 \Rightarrow w_1 s_{\tilde{\alpha}}), 2\rho \rangle$. Now consider a general-length path, where $w_2 = w_1 p$, and $p= s_{\tilde{\alpha}_1} \cdots s_{\tilde{\alpha}_n}$ corresponds to some minimal-distance path in $\QBG(W)$. We can iterate the above process to find the following.

    \vbls[-1]

    \begin{align*}
       \ell(w_2) &= \ell(w_1 s_{\tilde{\alpha}_1} \cdots s_{\tilde{\alpha}_n}) \\
        &= \ell(w_1 s_{\tilde{\alpha}_1} \cdots s_{\tilde{\alpha}_{n-1}}) + 1 - \langle \wt(w_1 s_{\tilde{\alpha}_1} \cdots s_{\tilde{\alpha}_{n-1}} \Rightarrow w_1 s_{\tilde{\alpha}_1} \cdots s_{\tilde{\alpha}_n}), 2\rho \rangle \\
        &= \cdots \\
        &= \ell(w_1) + n - \langle \wt(w_1 \Rightarrow w_1 s_{\tilde{\alpha}_1} \cdots s_{\tilde{\alpha}_n}), 2\rho \rangle \\
        &= \ell(w_1) + d(w_1 \Rightarrow w_2) - \langle \wt(w_1 \Rightarrow w_2), 2\rho \rangle.
    \end{align*}
    
    Rearranging then gives the desired result.
\end{proof}

The second ingredient we need to prove Thm. \ref{demlen} is the following:

\vbls

\begin{prop}\label{lplen}
    Let $x = w_x\varepsilon^{\mu_x} \in W_\mathcal{T}$, and $u \in \LP(x)$. Then:
    \[ \ell(x) = \langle u^{-1} \mu_x, 2\rho \rangle - \ell(u) + \ell(w_x u). \]
\end{prop}

\begin{proof}

    Let $w = s_{i_n} \cdots s_{i_1} \varepsilon^{\mu_x}$, such that $\ell(w) = n$ and $s_i = s_{\tilde{\beta}_i}$ for some simple affine root $\tilde{\beta}_i$. By the definition of the length function, we have the following.

    \vbls[-1]

    \begin{align*}
        \ell(x) = \ell(\varepsilon^{\mu_x}) + |\{\tilde{\alpha} \in \Inv(w) : \langle \mu_x, \tilde{\alpha} \rangle \geq 0 \}| - |\{\tilde{\alpha} \in \Inv(w) : \langle \mu_x, \tilde{\alpha} \rangle < 0 \}|.
    \end{align*}

    Let $u \in \LP(x)$. Then, as before, we have that:

    \[ \ell(\varepsilon^{\mu_x}) = 2\height(u^{-1}\mu_x) + 2|\{ \tilde{\alpha} \in \Phi^+ : \langle \mu_x, u\tilde{\alpha} \rangle = -1\}|. \]

    \vbls

    Note that $\langle \mu_x, u\tilde{\alpha} \rangle = -1 \implies u\tilde{\alpha} \in \Inv(w)$, and $\langle \mu_x, u\tilde{\alpha} \rangle \geq -1$ by length positivity. Hence:

    \vbls[-1]

    \begin{align*}
        |\{\tilde{\alpha} \in \Phi^+ : \langle \mu_x, u\tilde{\alpha} \rangle = -1 \}|
        &= |\{\tilde{\alpha} \in \Phi^+ \cap u^{-1}\Inv(w) : \langle \mu_x, u\tilde{\alpha} \rangle < 0 \}| \\
        &= |\{\tilde{\alpha} \in u\Phi^+ \cap \Inv(w) : \langle \mu_x, \tilde{\alpha} \rangle < 0 \}| \\
        &= |\{ \tilde{\alpha} \in \Inv(w) : \langle \mu_x, \tilde{\alpha} \rangle < 0 \}| - |\{ \tilde{\alpha} \in -u\Phi^+ \cap \Inv(w) : \langle \mu_x, \tilde{\alpha} \rangle < 0 \}|
    \end{align*}

    \vbls[-1.5]

    \begin{align*}
        \implies \ell(x) = 2\height(u^{-1}\mu_x) &+ 2|\{\tilde{\alpha} \in \Inv(w) : \langle \mu_x, \tilde{\alpha} \rangle < 0 \}| - 2|\{ \tilde{\alpha} \in -u\Phi^+ \cap \Inv(w) : \langle \mu_x, \tilde{\alpha} \rangle < 0 \}|\\
        &+ |\{\tilde{\alpha} \in \Inv(w) : \langle \mu_x, \tilde{\alpha} \rangle \geq 0 \}| - |\{\tilde{\alpha} \in \Inv(w) : \langle \mu_x, \tilde{\alpha} \rangle < 0 \}| \\
        = 2\height(u^{-1}\mu_x) &+ |\{\tilde{\alpha} \in \Inv(w) : \langle \mu_x, \tilde{\alpha} \rangle \geq 0 \}| + |\{\tilde{\alpha} \in \Inv(w) : \langle \mu_x, \tilde{\alpha} \rangle < 0 \}| \\
        &-2|\{ \tilde{\alpha} \in -u\Phi^+ \cap \Inv(w) : \langle \mu_x, \tilde{\alpha} \rangle < 0 \}| \\
        = 2\height(u^{-1}\mu_x) &+ |\Inv(w)| -2|\{ \tilde{\alpha} \in -u\Phi^+ \cap \Inv(w) : \langle \mu_x, \tilde{\alpha} \rangle < 0 \}|. \\
    \end{align*}

    \vbls[-0.5]

    We have that $|\Inv(w)| = \ell(w)$. Note also the following:

    \vbls[-0.5]

    \begin{align*}
        \tilde{\alpha} \in -u\Phi^+
        &\iff \tilde{\alpha} = -u\tilde{\beta}\ \textrm{for some}\ \tilde{\beta} \in \Phi^+ \\
        &\iff u^{-1}\tilde{\alpha} = -\tilde{\beta}\ \textrm{for some}\ \tilde{\beta} \in \Phi^+\\
        &\iff u^{-1}\tilde{\alpha} < 0.
    \end{align*}

    However, $\tilde{\alpha} \in \Inv(w) \implies \tilde{\alpha} \in \Phi^+$. Hence, $\tilde{\alpha} \in \Inv(u^{-1})$, and so $-u^{-1}\tilde{\alpha} > 0$. Applying length positivity to this element, we get the following.

    \vbls[-0.5]

    \begin{align*}
        \ell(x, u(-u^{-1}\tilde{\alpha}))
        = \ell(x, -\tilde{\alpha})
        = -\langle \mu_x, \tilde{\alpha} \rangle + 0 - 1\ \geq 0
        \implies \langle \mu_x, \tilde{\alpha} \rangle  \leq -1.
    \end{align*}

    Hence the condition that $\langle \mu_x, \tilde{\alpha} \rangle < 0$ in the above set is redundant. Using Prop. 2.4.6 of \cite{Welch19}, we therefore conclude the following.

    \vbls[-0.5]

    \begin{align*}
        \ell(x) &= 2\height(u^{-1}\mu_x) + |\Inv(w)| -2|\{ \tilde{\alpha} \in -u\Phi^+ \cap \Inv(w) : \langle \mu_x, \tilde{\alpha} \rangle < 0 \}| \\
        &= 2\height(u^{-1}\mu_x) + \ell(w) -2|\Inv(u^{-1}) \cap \Inv(w)| \\
        &= 2\height(u^{-1}\mu_x) + \ell(w) + (\ell(wu) - \ell(w) - \ell(u)) \\
        &= 2\height(u^{-1}\mu_x) - \ell(u) + \ell(wu).
    \end{align*}

    As required. 
\end{proof}

\begin{proof}[Proof of theorem \textup{\ref{demlen}}]  
    \hspace{3mm} Let $(u,v) \in M_{x,y}$, noting that this implies $v \in \LP(x*y)$. Using the definition of the Demazure product (which we assume to be well-defined), and using the formulae found in the above claims, we perform the following calculation, mimicking Schremmer's proof.

    \vbls[-1]

    \begin{align*}
        \ell(x*y)
        &= \langle v^{-1}( vu^{-1} \mu_x + \mu_y - v\wt(u \Rightarrow w_y v)), 2\rho \rangle - \ell(v) + \ell(w_{x*y} v) \\
        &= \langle u^{-1} \mu_x, 2\rho \rangle + \langle v^{-1} \mu_y, 2\rho \rangle - \langle \wt(u \Rightarrow w_y v), 2\rho \rangle - \ell(v) + \ell(w_x u) \\
        &= \ell(x) + \ell(u) - \ell(w_x u) + \ell(y) + \ell(v) - \ell(w_y v) - d(u \Rightarrow w_y v) - \ell(u) + \ell(w_y v) - \ell(v) + \ell(w_x u) \\
        &= \ell(x) + \ell(y) - d(u \Rightarrow w_y v).
    \end{align*}

    \vbls[-1]
\end{proof}

 As a corollary, we have the following important claim for discussing length additivity.

\begin{cor}\label{demlacon}
    Let $x, y \in W_\mathcal{T}$, and assume that the product $*$ is well-defined. Then:
    \[\ell(x*y) = \ell(x) + \ell(y)\ \iff\ |\LP(y) \cap w_y^{-1}\LP(x)| > 0.\]
\end{cor}

We also having the following result due to Muthiah and Pusk\'as. First, let our double affine roots be $\Phi_{\textup{daf}} = \{\widehat{\beta} = \tilde{\beta} + m\pi\ |\ \tilde{\beta} \in \Phi, m \in \mathbb{Z}\}$, with positive roots $\Phi_{\textup{daf}}^+ = \{ \sgn(m)(\tilde{\beta} + m\pi)\ |\ \tilde{\beta} \in \Phi^+, m \in \mathbb{Z} \}$. For $x \in \WTits$, we then define its inversion set to be $\Inv(x) := \{\hat{\beta} \in \Phi_{\textup{daf}}^+\ |\ x(\hat{\beta}) < 0 \}$, where $\WTits$ acts on $\Phi_{\textup{daf}}$ in the usual way, as an affinisation of equation \ref{eqn:w-action}. Note that these inversion sets are typically infinite, but their intersections are finite, as seen in the following.

\begin{prop}[\cite{Muthiah-Puskas24}, Thm. 5.4]\label{mutpus}
    Let $x, y \in W_{\mathcal{T}}$. Then:
    \[\ell(xy) = \ell(x) + \ell(y) - 2|\{\Inv(x) \cap \Inv(y^{-1})\}|.\]
\end{prop}

From this, it follows that the product is length additive when the corresponding intersection of inversion sets is empty. Using this, we now prove Thm. \ref{bigone}.

\subsection{Proof of Thm. \ref{bigone}}

``$\Rightarrow$": Assume $\ell(xy) \neq \ell(x) + \ell(y)$. Then by Prop. \ref{mutpus}, there must exist some positive double affine root $\widehat{\beta}[m] = \sgn(m)(\tilde{\beta} + m\pi) \in \Inv(x) \cap \Inv(y^{-1}) \subseteq \Phi_{\textrm{af}}^+$, where $\tilde{\beta} \in \Phi^+$. Hence, $x(\widehat{\beta}) < 0$ and $y^{-1}(\widehat{\beta}) < 0$. As usual, let $x = w_x \varepsilon^{\mu_x}, y = w_y \varepsilon^{\mu_y}$. Then, the following inequalities must hold.

\vbls[-0.5]

\begin{align*}
    |m| - \langle \mu_x, \sgn(m)\tilde{\beta} \rangle &\leq -\Phi^+(\sgn(m)w_x\tilde{\beta}), \\
    |m| + \langle \mu_y, \sgn(m)w_y^{-1}\tilde{\beta} \rangle &\leq -\Phi^+(\sgn(m)w_y^{-1}\tilde{\beta}).
\end{align*}

Now let $v \in w_y^{-1} \LP(x)$, so that $\ell(x, w_yv\tilde{\alpha}) \geq 0$ for every $\tilde{\alpha} \in \Phi^+$. Let $\tilde{\gamma} = \sgn(m)v^{-1}w_y^{-1}\tilde{\beta} \in \Phi$. Then:

\vbls[-1.5]

\begin{align*}
    \ell(x, w_y v \tilde{\gamma})
    &= \langle \mu_x, \sgn(m) \tilde{\beta} \rangle + \Phi^+(\sgn(m)\tilde{\beta}) - \Phi^+(\sgn(m)w_x\tilde{\beta}) \\
    &\geq |m| + \Phi^+(\sgn(m)w_x\tilde{\beta}) + \Phi^+(\sgn(m)\tilde{\beta}) - \Phi^+(\sgn(m)w_x\tilde{\beta}) \\
    &= |m| + \Phi^+(\sgn(m)\tilde{\beta})\ \geq\ 1.
\end{align*}

\vbls[0.5]

Hence $\tilde{\gamma} \in \Phi^+$, otherwise $\ell(x, w_y v \tilde{\alpha}) \leq -1$ for some positive root $\tilde{\alpha} = -\tilde{\gamma}$. But then:

\vbls[-0.5]

\begin{align*}
    \ell(y, v\tilde{\gamma})
    &= \langle \mu_y, \sgn(m)w_y^{-1}\tilde{\beta} \rangle + \Phi^+(\sgn(m)w_y^{-1}\tilde{\beta}) - \Phi^+(\sgn(m)\tilde{\beta}) \\
    &\leq -|m| - \Phi^+(\sgn(m)w_y^{-1}\tilde{\beta})+ \Phi^+(\sgn(m)w_y^{-1}\tilde{\beta}) - \Phi^+(\sgn(m)\tilde{\beta}) \\
    &= -(|m| - \Phi^+(\sgn(m)\tilde{\beta}))\ \leq\ -1.
\end{align*}

And so $v \notin \LP(y)$. Therefore $w_y^{-1} \LP(x) \cap \LP(y)$ is empty, and by Cor. \ref{demlacon}, $\ell(x*y) \neq \ell(x) + \ell(y)$.

\vbls

``$\Leftarrow$": Assume $\ell(x*y) \neq \ell(x) + \ell(y)$. Then $w_y^{-1}\LP(x) \cap \LP(y)$ must be empty. Let $v \in \LP(y)$, so that $\ell(y, v\tilde{\alpha}) \geq 0\ \forall \tilde{\alpha} \in \Phi^+$. We must also have that $v \notin w_y^{-1}\LP(x)$, so there exists $\tilde{\beta} \in \Phi^+$ such that $\ell(x, w_yv\tilde{\beta}) < 0$. 

Note that for this $\tilde{\beta}$, we can assume that $\ell(y, v\tilde{\beta}) \geq 1$, otherwise we can make an adjustment and use $vs_{\tilde{\beta}}$ instead of $v$, which would still be length positive for $y$ and reduce the number of inversions caused by $w_yv$. If no such $\tilde{\beta}$ existed, we could then repeatedly adjust until we found a $v' \in w_y^{-1}\LP(x) \cap \LP(y)$, contradicting our assumptions.

These facts give us the following inequalities.

\vbls[-0.5]

\begin{align*}
    \langle \mu_y, v\tilde{\beta} \rangle + \Phi^+(v\tilde{\beta}) - \Phi^+(w_yv\tilde{\beta}) &\geq 1,\\
    \langle \mu_x, w_yv\tilde{\beta} \rangle + \Phi^+(w_yv\tilde{\beta}) - \Phi^+(w_xw_yv\tilde{\beta}) &\leq -1.
\end{align*}

Let $\widehat{\gamma} = \sgn(m)(\tilde{\gamma} + m\pi) \in \Phi_{\textrm{af}}^+$, for some $\tilde{\gamma} \in \Phi^+$. Then for $x(\widehat{\gamma}) < 0$ and $y^{-1}(\widehat{\gamma}) < 0$, we need the following inequalities to hold.

\vbls[-1]

\begin{align*}
    |m| - \langle \mu_x, \sgn(m)\tilde{\gamma} \rangle &\leq -\Phi^+(\sgn(m)w_x\tilde{\gamma}), \\
    |m| + \langle \mu_y, \sgn(m)w_y^{-1}\tilde{\gamma} \rangle &\leq -\Phi^+(\sgn(m)w_y^{-1}\tilde{\gamma}).
\end{align*}

Let $\tilde{\gamma} = -\sgn(m)w_y v \tilde{\beta}$, with $\sgn(m)$ chosen such that $\tilde{\gamma} \in \Phi^+$. Then we can manipulate these inequalities to become the following.

\vbls[-1]

\begin{align*}
    |m| + \langle \mu_x, w_yv\tilde{\beta} \rangle &\leq \Phi^+(w_xw_yv\tilde{\beta}) - 1, \\
    |m| - \langle \mu_y, v\tilde{\beta} \rangle &\leq \Phi^+(v\tilde{\beta}) - 1.
\end{align*}

Setting $m = -\Phi^+(w_yv\tilde{\beta})$, which ensures $\tilde{\gamma} \in \Phi^+$, we then obtain inequalities which are guaranteed by the earlier length positivity conditions, and so $\widehat{\gamma} \in \Inv(x) \cap \Inv(y^{-1})$, and hence $\ell(xy) \neq \ell(x) + \ell(y)$ as required.
\vbls
\section{Examples}\label{sec:examples}

\subsection{Length Additive Pair} We compute the Demazure product of the elements $x = \varepsilon^{\delta + \Lambda_0}s_0s_1, y = \varepsilon^{4\delta + \Lambda_0}$, aligning with \cite{Muthiah-Puskas24} section 6.1, with changes made to account for differences in convention (cf. \cite{Muthiah-Puskas24} Equation 2.9). In their convention, they compute $T_{x}T_y = T_{xy}$ in the Hecke algebra for this chosen $x, y$. Hence, our goal is to show that $x*y = xy$.

We rewrite $x$ in the form $w\varepsilon^\mu$ using the action of $W$ on $\mathcal{T}$ \cite{KacIDLA}.

\vbls[-0.5]

\begin{equation}
    w_0\tau^{\lambda}(\mu_0 + m\delta + l\Lambda_0) = w_0\mu_0 + lw_0\lambda + (m - \langle \mu_0, \lambda \rangle - \frac{l}{2}\langle \lambda, \lambda \rangle)\delta + l\Lambda_0.
\end{equation}

This gives $x = s_0s_1 \varepsilon^{s_1s_0(\delta + \Lambda_0)} = \tau^{\alpha^{\vee}} \varepsilon^{-\alpha^{\vee} + \Lambda_0}$. We now compute $\LP(x)$ and $\LP(y)$ using Prop. \ref{prop:LP-algorithm}. Following the previous notation, for $\LP(x)$ we have $l = 1, k=-1$ which gives $t = -1$ and $j = 0 = \pm\frac{l-1}{2}$. Therefore $t \leq 0$, $t \leq \Phi^+(w_0\alpha) - r - 1 = -1$, and $-r \leq t \leq -1$, so we have $\LP(x) = \{\tau^{-\alpha^{\vee}}, \tau^{-\alpha^{\vee}}s_1, \tau^{-\alpha^{\vee}}s_0 \} = \{s_1s_0, s_1s_0s_1, s_1\}$ from steps 3 and 5 respectively.

For $y = \varepsilon^{4\delta + \Lambda_0}$, we have $k = 0, r = 0, l = 1, j = 0 = \pm\frac{l-1}{2}$, and so $t = 0$. We therefore have $t \leq 0$ and $t \leq \Phi^+(w_0\alpha) - r - 1 = 0$, so from step 3 we have $\LP(y) = \{\tau^{0\alpha^{\vee}}, \tau^{0\alpha^{\vee}}s_1\} = \{e, s_1\}$.

Hence $w_y\LP(y) = \{e, s_1\}$ and so we have a non-empty intersection $\LP(x) \cap w_y\LP(y) = \{s_1\}$. Hence, by Thms. \ref{bigone} and \ref{demlen}, we should expect $x*y = xy$, as computed by Muthiah and Pusk\'as. We have $M_{x,y} = \{(s_1, s_1)\}$ with $d(s_1 \Rightarrow w_ys_1) = 0$, and we compute the Demazure product using the notation of Defn. \ref{def:gen-dem-prod}:

\vbls[-1]

\begin{align*}
    x * y\ 
    =\ w_x uv^{-1} \varepsilon^{vu^{-1}\mu_x + \mu_y - v\wt(p)} 
    &= s_0 s_1 s_1 s_1^{-1} \varepsilon^{s_1 s_1^{-1} (-\alpha^{\vee} + \Lambda_0) + (4\delta + \Lambda_0) - s_1(0)} \\
    &= s_0 s_1 \varepsilon^{-\alpha^{\vee} + 4\delta + 2\Lambda_0} \\
    &= s_0 s_1 \varepsilon^{-\alpha^{\vee} + \Lambda_0} \varepsilon^{4\delta + \Lambda_0}\
    =\ xy.
\end{align*}

\subsection{Non-Length Additive Pair} Now let $x = \varepsilon^{\Lambda_0}s_0s_1s_0, y = \varepsilon^{\Lambda_0}s_0$ as per \cite{Muthiah-Puskas24}, section 6.2. In their notation, they compute the following:

\begin{equation*} T_x T_y = qT_{xy} + (q-1)T_{\pi^{2\Lambda_0 -\delta - \tilde{\alpha}_0}s_0} \end{equation*}

And so $T_x T_y \equiv -T_{\pi^{2\Lambda_0 -\delta - \tilde{\alpha}^{\vee}_0}s_0} \mod q$. Hence, following Conjecture \ref{conj:mp} and converting to our notation, we wish to show that $x * y = s_0\varepsilon^{2\Lambda_0 - \delta - \tilde{\alpha_0}^{\vee}}$. 

\vbls

First we rewrite these elements as before, giving $x = s_1\tau^{-2\alpha^{\vee}}\varepsilon^{2\alpha^{\vee} - 4\delta + \Lambda_0}$ and $y = s_1\tau^{-\alpha^{\vee}}\varepsilon^{\alpha^{\vee} - \delta + \Lambda_0}$. For $x$, we have $w_0 = s_1, r = -2, k = 1, l = 1$ giving $t = 2$ and $j = 0 = \pm\frac{l-1}{2}$. By Prop. \ref{prop:LP-algorithm}, we find that $\LP(x) = \{\tau^{2\alpha^{\vee}}, \tau^{2\alpha^{\vee}}s_1, \tau^{2\alpha^{\vee}}s_1s_0\} = \{s_0s_1s_0s_1, s_0s_1s_0, s_0s_1\}$. For $y$, we have $w_0 = s_1, r = -1, k = 1, l = 1$ giving $t = 1$ and $j = 0 = \pm\frac{l-1}{2}$, and hence $\LP(y) = \{\tau^{\alpha^{\vee}}, \tau^{\alpha^{\vee}}s_1, \tau^{\alpha^{\vee}}s_1s_0\} = \{s_0s_1, s_0, e\}$, and so $w_y\LP(y) = \{s_1, e, s_0\}$. 

\vbls

We now have $\LP(x) \cap w_y\LP(y) = \emptyset$, and so a shortest path will have length at least one. It is easy to see that $s_0s_1 \Rightarrow s_0$ is a path of length exactly one, and therefore minimal. Hence $(s_0s_1, w_y^{-1}s_0) = (s_0s_1, e) \in M_{x,y}$, with $\wt(s_0s_1 \Rightarrow s_0) = \tilde{\alpha}_1^{\vee} = \alpha^{\vee}$. Using Defn. \ref{def:gen-dem-prod} with this choice of distance minimising pair, we find that:

\vbls[-0.5]

\begin{align*}
    x*y
    = w_xuv^{-1}\varepsilon^{vu^{-1}\mu_x + \mu_y - v\wt(p)}
    &= s_1 \tau^{-2\alpha^{\vee}} \tau^{\alpha^{\vee}} e \varepsilon^{\tau^{-\alpha^{\vee}}(2\alpha^{\vee} - 4\delta + \Lambda_0) + \alpha^{\vee} - \delta + \Lambda_0 - e(\alpha^{\vee})} \\
    &= s_1\tau^{-\alpha^{\vee}}\varepsilon^{\alpha^{\vee} - 2\delta + 2\Lambda_0} \\
    &= s_0\varepsilon^{-\tilde{\alpha_0}^{\vee} - \delta + 2\Lambda_0}.
\end{align*}

This aligns with the Muthiah and Pusk\'as' result, taking the $q = 0$ term in their Hecke product. It can also be verified that $\ell(x*y) = \ell(x) + \ell(y) - 1$, an example of Thm. \ref{demlen}.

\vbls[2]
\bibliography{references}{}
\bibliographystyle{amsalpha-edited.bst}

\end{document}